\documentclass[11pt,a4paper]{article}
\usepackage{amsfonts}
\usepackage{amsmath, amssymb, amsthm}
\usepackage[dvips]{graphicx}
\usepackage{graphics}
\usepackage{adjustbox} % Used to constrain images to a maximum size 
\usepackage{color}
\usepackage{calc}
\usepackage{float}
\usepackage{tikz}
\usetikzlibrary{decorations.shapes}
\usetikzlibrary{arrows,positioning} 
\usetikzlibrary{decorations.markings}
\tikzstyle{vertex}=[circle, draw, inner sep=0pt, minimum size=6pt]
\usepackage{eufrak}
\usepackage{grffile} % extends the file name processing of package graphics 
\usepackage{hyperref}
    \usepackage{longtable} % longtable support required by pandoc >1.10
    \usepackage{booktabs}  % table support for pandoc > 1.12.2
\newcommand{\E}{\mathcal{E}}
\newcommand{\supp}{\textrm{supp}}

\newcommand{\aff}{\textrm{Aff}}

\newtheorem{theorem}{Theorem}
\newtheorem{assumption}[theorem]{Assumption}

\newtheorem{corollary}[theorem]{Corollary}

\newtheorem{definition}[theorem]{Definition}
\newtheorem{example}[theorem]{Example}

\newtheorem{lemma}[theorem]{Lemma}
\newtheorem{notation}[theorem]{Notation}

\newtheorem{proposition}[theorem]{Proposition}
\newtheorem{remark}[theorem]{Remark}

\numberwithin{equation}{section}
\numberwithin{theorem}{section}

\tikzset{
    >=stealth',
    punkt/.style={
           rectangle,
           rounded corners,
           draw=black, very thick,
           text width=6.5em,
           minimum height=2em,
           text centered},
    pil/.style={
           ->,
           thick,
           shorten <=2pt,
           shorten >=2pt,}
}

\tikzset{decorate sep/.style 2 args=
{decorate,decoration={shape backgrounds,shape=circle,shape size=#1,shape sep=#2}}}

\begin{document}

\title{On the support of extremal martingale measures \\ with given marginals: the countable case\thanks{This work is partially supported by the ANR project ISOTACE (ANR-12-MONU-0013). We are also grateful to Beatrice Acciaio, Peter Allen, Graham Brightwell, Alex Cox, Davide Gabrielli, Riccardo Pallottini and Fr\'ed\'eric Patras for their helpful remarks.}}

\author{Luciano Campi\footnote{London School of Economics and Political Science, Department of Statistics, United Kingdom.} \hspace{2mm}  Claude Martini\footnote{Zeliade Systems, France.}}

\date{}

\maketitle

\begin{abstract}
We investigate the supports of extremal martingale measures with pre-specified marginals in a two-period setting. First, we establish in full generality the equivalence between the extremality of a given measure $Q$ and the denseness in $L^1(Q)$ of a suitable linear subspace, which can be seen in a financial context as the set of all semi-static trading strategies. Moreover, when the supports of both marginals are countable, we focus on the slightly stronger notion of weak exact predictable representation property (henceforth, WEP) and provide two combinatorial sufficient conditions, called ``2-link property'' and ``full erasability'', on how the points in the supports are linked to each other for granting extremality. When the support of the first marginal is a finite set, we give a necessary and sufficient condition for the WEP to hold in terms of the new concepts of $2$-net and deadlock. Finally, we study the relation between cycles and extremality.\medskip\\
\emph{Keywords and phrases:} model-free pricing, extremal measures, martingale optimal transport, weak predictable representation property, cycles. \vspace{1mm} \\
\emph{MSC Classification 2010:} 60G42, 91G80.
\end{abstract}

\section{Introduction}

In this paper we study the supports of extremal martingale measures, defined on the product space $\mathbb R_+ ^2 = [0,\infty)^2$ equipped with its Borel $\sigma$-field, under the constraints of having given marginals $\mu$ and $\nu$. The set of all such measures, which is nonempty if and only if $\mu$ is smaller than $\nu$ in the convex order, is at the core of \emph{martingale optimal transport}, a new field of research that has been introduced by \cite{BHLP} in the discrete-time case and by \cite{GHLT} in the continuous-time case. The martingale optimal transport problem is a variant of the classical Monge-Kantorovich optimal transport problem (see \cite{villani}), and it consists in optimizing a given functional over the set $\mathcal M(\mu,\nu)$ of all probability measures with pre-specified marginals $\mu$ and $\nu$ and satisfying the martingale property. The latter property is what makes the difference with the classical optimal transport problem and it is motivated by financial applications. An important growing literature originated from the seminal paper \cite{hobson}, which started the model-free approach to derivative pricing using techniques based on the Skorokhod embedding problem. Within this approach, only very weak assumptions are made, namely the price process of the underlying is a martingale (to rule-out arbitrage opportunities) and its marginals are given by the observation of European Call prices (via the so-called Breeden-Litzenberger formula\footnote{We recall that the Breeden-Litzenberger formula (cf. \cite{breeden}) states that from the prices of European Call options $C(K,T) = \mathbb E_Q [(S_T -K)_+]$ with a fixed maturity $T$ and for all strike prices $K>0$, one can deduce the law of the underlying $S_T$ under $Q$. Indeed taking the right-derivative in $K$ gets $\frac{\partial}{\partial K+}C(T,K) = -Q(S_T >K)$, for all $K>0$.}). Hence, computing for instance the super-replication price of some derivative boils down to maximizing the expected value of its pay-off, say $f$, over the set $\mathcal M(\mu,\nu)$, yielding the following martingale transport problem:
\begin{equation}\label{sup} \sup_{Q \in \mathcal M(\mu,\nu)} Q(f),\end{equation}
where $Q(f)$ denotes the expectation of $f$ under $Q$. This problem has been studied in great depth in \cite{BJ} for a large class of payoffs. The results therein have been further generalized in \cite{HLT}. In the papers \cite{HN,HK} the (martingale) optimal transport has been found for $f(x,y) = \pm |x-y|$. 

Our interest in the extremal elements of the set $\mathcal M(\mu,\nu)$ is fundamentally motivated by the model-free approach. Indeed, the notion of extremal measures is intimately related to that of optimizers and hence to model-free derivatives pricing. More precisely, Bauer Maximum Principle (cf. \cite[Sec. 7.69]{AB}) states that any upper semi-continuous convex function on a compact convex subset of a locally convex Hausdorff space has a maximizer that is an extremal point. Hence it applies to optimization problems such as \eqref{sup} whenever the pay-off $f$ is regular enough. Moreover, we note that when the maximizer is unique it is necessarily an extremal point. Therefore, understanding the support of extremal measures can give insights on the solutions of martingale optimal transport problems such as (\ref{sup}). Another motivation for this study comes as a consequence of our first result (cf. Theorem \ref{douglas2}), which roughly states the following equivalence:\smallskip

\emph{A martingale measure $Q \in \mathcal M(\mu,\nu)$ is extremal if and only if every derivative can be approximately replicated on the support of $Q$ by semi-static strategies}.\smallskip

This can be seen as an extension, to the model-free setting, of the well-known equivalence in the classical setting between ``market completeness'' and extremality of $Q$ in the set of all martingale measures without constraints on the marginals (see, e.g. \cite{JS} for the discrete-time case), which is in turn the financial translation of one of the most important results in martingale theory, namely that extremality is equivalent to the predictable representation property (see \cite{dell} in discrete-time and, e.g., \cite[Theorem 4.7, Ch. V]{RY} in continuous-time).

Therefore, in the model-free setting, knowing the support of extremal measures in $\mathcal M(\mu,\nu)$ gives a way to generate models where any derivative can be (approximately) replicated by semi-static strategies. 

Besides the financial motivation, the problem of characterizing the supports of extremal measures is mathematically interesting on its own and it has quite a long history. Indeed, there is a rich literature on the support of extremal probability measures with given marginals (without the martingale property), which goes back to a paper by Birkhoff \cite{birk}, where a complete description of extremal measures is given in the finite case, i.e. both marginals have finite supports. The main result therein establishes that a probability measure with given marginals is extremal if and only if its support does not contain any cycle. Many papers followed, e.g. \cite{BS,BC,denny,douglas,HW,klopo1,klopo2,letac,lindenstrauss,muk} among others, giving different kinds of characterizations in the finite or countable case and going from functional analysis to combinatorics. In particular \cite{denny} extends to the countable case Birkhoff's result about absence of cycles in the support of extremal measures. In the general case, the problem of giving a complete description of extremal measures with given marginals is still open. 

Inspired by this literature, the present paper provides, in full generality, a characterization of extremality in the martingale case in terms of a weaker form of the predictable representation property. In the more specific case of marginals with countable supports, we define a slightly stronger property (called WEP) which allows us to focus on the combinatorial properties of the support of a given measure in $\mathcal M(\mu,\nu)$. Therefore, we propose two sufficient conditions, called ``2-link property'' and ``full erasability'', having a strong combinatorial flavour. Three important examples satisfy those criteria and hence they are extremal measures: the binomial tree, Hobson and Klimmek's trinomial tree (cf. \cite{HK}) and the left curtain introduced in \cite{BJ} (at least in the case when one of the two marginals has finite support). Those criteria are very easy to implement for generating many other examples of extremal supports. Moreover, we introduce the new notions of 2-net and deadlock, which allow to formulate an essentially necessary and sufficient condition for the WEP. Finally, we also investigate to which extent a characterization of extremality in terms of absence of cycles (compare \cite{letac,muk}) is possible in the martingale setting. \medskip

The paper is organized as follows: Section \ref{sec:setting} sets the framework and the main notation, while in Section \ref{sec:douglas} we give a characterization of extremality in terms of a weak predictable representation property. In Section \ref{sec:WEP} we introduce the weak \emph{exact} predictable representation property (WEP). Section \ref{suff} contains the two sufficient conditions with examples. Moreover, in Section \ref{q-affine-functions-and-2-nets} we study the relation between 2-nets, deadlocks and extremality. Finally, Section \ref{cycles} focuses on the relation between cycles, extremality and WEP, and Section \ref{conclusion} concludes the paper.

%%%%%%%%%%%%%%
\section{Setting and notation}\label{sec:setting}
%%%%%%%%%%%%%%

Let $\mu$ and $\nu$ be two probability laws 
on $(\mathbb R_+ , \mathcal B(\mathbb R_+))$, where $\mathbb R_+ := [0,\infty)$ denotes the set of all positive real numbers.
Let $\mathcal P(\mu,\nu)$ denote the set of all probability measures on $(\mathbb R_+ ^2 , \mathcal B(\mathbb R_+ ^2))$ with marginals $\mu$ and $\nu$, i.e.
\[ Q(A \times \mathbb R_+) =\mu(A), \quad Q(\mathbb R_+ \times A) =\nu(A), \quad \textrm{for all }  A \in \mathcal B(\mathbb R_+). \]
For any $Q \in \mathcal P(\mu,\nu)$, the following decomposition holds:
\[ Q(dx,dy) = q(x,dy) \mu(dx) ,\]
where $q(x,dy)$ is a probability kernel. We will always work under the following
assumption:
\begin{assumption}\label{main-ass} Let $\int x \mu(dx) = \int y \nu(dy) =1$ and $\mu \preccurlyeq \nu$ in the convex order, i.e.
\[ \int c(x) \mu (dx) \le \int c(y) \nu(dy),\]
for all convex functions $c: \mathbb R_+ \to \mathbb R$.
\end{assumption}
Let $\mathcal M(\mu,\nu)$ denote the set of all probability measures in $\mathcal P(\mu,\nu)$ with the martingale property, i.e.
\[ \int_{\mathbb R_+} y q(x,dy) = x ,\quad \mu-a.e.\] 
Assumption \ref{main-ass} implies that $\mathcal M(\mu,\nu)$ is nonempty (cf. \cite{kellerer, strassen}). Moreover the set $\mathcal M(\mu,\nu)$ is compact for the weak convergence of measures (cf. Proposition 2.4 in \cite{BHLP}). 

The central notion of this paper is the one of extremal point of $\mathcal M(\mu,\nu)$, which is any probability measure $Q \in \mathcal M(\mu,\nu)$ such that if $Q=\alpha Q_1 + (1-\alpha) Q_2$ for some $\alpha \in (0,1)$ and $Q_i \in \mathcal M(\mu,\nu)$, $i=1,2$, then $Q=Q_1=Q_2$. The fact that $\mathcal M(\mu,\nu)$ is weakly compact yields that the set of its extremal points is nonempty (cf. Corollary 7.66 in \cite{AB}). When there is no ambiguity, ``$Q$ extremal'' will mean ``$Q$ extremal in  $\mathcal M(\mu,\nu)$''.

Finally, for any real-valued measurable function $f$ defined on some probability space $(\Omega, \mathcal F, Q)$ we will use indifferently the notations $\mathbb E_Q (f)=Q(f)=\int fdQ = \int f(x,y) dQ(x,y)$ for the expectation of $f$ under $Q$.

\begin{remark}
\emph{The setting admits the usual model-free finance interpretation as follows: let $(x,y)$ be a generic element of the sample space $\mathbb R_+ ^2$, and let $X$ (resp. $Y$) denote the application $X(x,y)=x$ (resp. $Y(x,y)=y$). Hence $(X,Y)$ is a two-dimensional random vector defined on the measurable space $(\mathbb R_+ ^2, \mathcal B(\mathbb R_+ ^2))$. Under any measure $Q \in \mathcal M(\mu,\nu)$, $X$ and $Y$ have respective laws $\mu$ and $\nu$. Moreover, under Assumption \ref{main-ass}, $(1,X,Y)$ is martingale under $Q$ for its natural filtration. Hence it can be viewed as the (discounted) price process of some risky asset. Moreover any measure $Q \in \mathcal M(\mu,\nu)$ specifies the full law of the price process, in other terms gives a price model for the risky asset, which is compatible with the knowledge of the marginals $\mu,\nu$ as well with the absence of arbitrage opportunities (due to the martingale property).}
\end{remark}

%%%%%%%%%%%%%%%%%%%%%%%%%%%%%%%%%%%%
\section{The Douglas-Lindenstrauss-Naimark Theorem and its consequences}\label{sec:douglas}
%%%%%%%%%%%%%%%%%%%%%%%%%%%%%%%%%%%%

In this section, we give a functional analytical characterisation of extremality in $\mathcal M(\mu,\nu)$ with a natural financial interpretation. We stress that the results in this part of the paper hold in full generality. More restrictive assumptions on the supports of the marginals $\mu$ and $\nu$ will be needed in the next sections.
We start with recalling the following classical result relating extremality and denseness of subspaces of $L^1 (Q)$.  

\begin{theorem}[Douglas \cite{douglas}, Lindenstrauss \cite{lindenstrauss}, Naimark \cite{naimark}] \label{douglas}
Let $(\Omega, \mathcal F , Q)$ be a probability space and let $F$ be a linear subspace of $L^1 (Q)$ such that $1 \in F$. The following are equivalent:\begin{enumerate}
\item[(i)] $Q$ is an extremal point of the set of all probability measures $R$ on $(\Omega, \mathcal F)$ (not necessarily equivalent to $Q$) such that $\mathbb E_R (f) = \mathbb E_Q (f)$ for all $f \in F \cap L^1 (R)$;
\item[(ii)] $F$ is dense in $L^1 (Q)$ with the strong topology.
\end{enumerate}
\end{theorem}

\begin{remark}
\emph{The Douglas-Lindenstrauss-Naimark theorem was used in \cite[Ch. V]{RY} to prove the predictable representation property (PRP) for continuous martingales in the Brownian filtration. Other applications of this theorem in relation with various notions of market completeness can be found in \cite{campiRendi,campiDEF}.}
\end{remark}

A direct application of this theorem to our setting gives the following equivalence, where the notation $L^0 (\mu)$ stands for the set of all measurable functions with finite values $\mu$-a.s.

\begin{theorem}\label{douglas2} Let $Q \in \mathcal M(\mu,\nu)$. The following two properties are equivalent:
\begin{enumerate}
\item[(i)] $Q$ is extremal in $\mathcal M(\mu,\nu)$;

\item[(ii)] the weak Predictable Representation Property (PRP) holds in the following sense: the set of all functions $f \in L^1(Q)$ that can be represented as
\begin{equation}\label{weakPRP} f(x,y) = \varphi(x) + h(x) (y-x) - \psi (y), \quad Q-\textrm{a.s.} \end{equation}
for some functions $\varphi \in L^1 (\mu), \psi \in L^1(\nu)$ and $h \in L^0 (\mu)$, is dense in $L^1 (Q)$.
\end{enumerate}
\end{theorem}

\begin{proof} We prove first that weak PRP under $Q \in \mathcal M(\mu,\nu)$ implies that $Q$ is extremal. Assume that $Q=\alpha Q_1 + (1-\alpha)Q_2$ for some $\alpha \in (0,1)$ and $Q_i \in \mathcal M (\mu,\nu)$ for $i=1,2$. Therefore, $Q_i \ll Q$ for $i=1,2$. Consider any functions $f \in L^1(Q)\cap L^1(Q_i)$ such that $f(x,y)=\varphi(x)+ h(x)(y-x)-\psi(y)$ for suitable functions $h,\varphi,\psi$ and $Q$-a.s., hence $Q_i$-a.s. as well. Taking the expectation under those measures, we get $Q(f)=\mu(\varphi)+\nu(-\psi)=Q_i (f)$ for $i=1,2$, for all bounded functions $f$ satisfying (\ref{weakPRP}). By denseness, we obtain that $Q=Q_1=Q_2$, i.e. $Q$ is extremal in $\mathcal M(\mu,\nu)$. It remains to prove the converse. To do so, it suffices to apply Theorem \ref{douglas} to the set
\begin{eqnarray}\label{F}
F &=& \{ f \in L^1 (Q) : f(x,y) = \varphi(x) + h(x) (y-x) - \psi(y), \nonumber \\
&&\textrm{ for some } \varphi \in L^1 (\mu), \psi \in L^1(\nu), h \in L^0(\mu)\},
\end{eqnarray}
which clearly contains the function $1$.
Indeed, notice that with such a choice of the subspace $F$, the condition $\int f dR = \int f dQ$ for all $f \in F \cap L^1 (R)$ implies that $R \in \mathcal M(\mu,\nu)$.\end{proof}

Theorem \ref{douglas2} has a natural financial interpretation. Any measure $Q \in \mathcal M(\mu,\nu)$ can be seen as some price model consistent with the marginals $\mu,\nu$. Hence, the theorem above gives that the extremal models are the ones where any contingent claim can be (approximately) replicated by trading dynamically in the underlying in 
a self-financing way and statically in some European options with payoffs $\varphi$ (at time $1$) and $\psi$ (at time 2). 

We conclude this section with a proposition yielding in particular that the extremal points of $\mathcal M(\mu,\nu)$ are fully characterized by their supports.

\begin{proposition}\label{absolute}
$Q \in \mathcal M(\mu,\nu)$ is extremal if and only if for any $R \in \mathcal M (\mu,\nu)$ with $R \ll Q$ we have $R=Q$ .
\end{proposition}

\begin{proof}
Let $Q \in \mathcal M(\mu,\nu)$. Assume that $Q$ is not extremal, i.e. $Q=\alpha Q_1+(1-\alpha) Q_2$ for some $Q_i \in \mathcal M(\mu,\nu)$, $i=1,2$, with $Q_1 \ne Q$ and some $\alpha\in (0,1)$. Hence we have $Q_1 \ll Q$. Now let $Q$ be extremal and $R \in \mathcal M (\mu,\nu)$ such that $R \ll Q$. 
Denote \(\ell= \frac{dR}{dQ}\) the corresponding density given by the Radon-Nikodym Theorem. Moreover we can choose a version of $\ell$ in $L^\infty(Q)$, yielding $L^1(Q) \subset L^1(R)$. Hence, for any $f \in L^1(R)$ with representation triple $(h, \varphi, \psi)$ as in (\ref{weakPRP}) one has
\[ \mu (\varphi ) - \nu (\psi ) =\int f dQ = \int fdR = \int \ell f dQ, \] 
It follows that  $\ell-1$ is orthogonal to $F$, which is dense in $L^1(Q)$. Therefore we have $\ell=1$, so that $R=Q$.\end{proof}

%%%%%%%%%%%%%%%%%%%%%%%%%%%%%%%%
\section{The Weak Exact Predictable Representation Property%(WEP)
}\label{sec:WEP}
%%%%%%%%%%%%%%%%%%%%%%%%%%%%%%%%

From now on we will work under the following standing assumption:
\begin{assumption}\label{main-ass2} $\mu$ and $\nu$ are supported on countable subsets $X$ and $Y$ of $\mathbb R_+$.
\end{assumption}

We introduce some notations in this discrete support context. Let $S$ be any subset of $X \times Y$. For $(x, y) \in S$ we let
\[Y_S(x) = \{z \in Y : (x, z) \in S\}, \quad X_S(y) = \{t \in X : (t, y) \in S\},\]
and
\[S_X = \{x \in X : \exists y \in Y, (x, y) \in S\}, \quad S_Y = \{y \in Y: \exists y \in Y, (x, y) \in S\}.\]
We call \emph{mesh} any set $S \subset X \times Y$ such that $| S_X | = 1$. If in addition $S$ satisfies $| S_Y | = 2$, it will be called {\it 2-mesh} (or binomial mesh).
For any measure $Q \in \mathcal P(\mu,\nu)$ we define its support as the set
\[ \supp (Q) := \{(x,y) \in X \times Y : Q(x,y)>0 \}.\] 
Whenever $S$ is the support of a probability $Q$, i.e. $S =\supp (Q)$, and when there is no ambiguity we will drop $S$ from the notation
$X_S(y), Y_S(x)$ and simply write $X(y), Y(x)$.

Finally, the notation $ | A |$ denotes the cardinality of any set $A$ and by numbering (or ordering) of any countable set $A$ we mean any possible representation of the set as a sequence. 

\vspace{0.5cm}

We introduce now the following {\it exact} strengthening of the weak PRP, which is motivated by Mukerjee purely geometrical characterization of extremality in the non-martingale case (see the introduction in \cite{muk} and his Theorem 2.7). 
Such a property is easier to handle than the weak PRP (cf. Theorem \ref{douglas2}) since it avoids the issues of integrability of the functions appearing in the representation, as well as the passage to the limit in $L^1 (Q)$.
Yet the main difference with respect to the weak PRP is subtler: we require the replication property only for functions defined on a subset $S$ of the product space $X \times Y$. In most instances, $S$ will be the support of some measure $Q$.

\begin{definition}[Weak Exact PRP] Let $S$ be a subset of $X \times Y$. We say that the Weak Exact PRP (henceforth, the WEP) holds for $S$ 
if for every function $f:S \to \mathbb R$, there exist functions $h,\varphi : S_X \to \mathbb R$ and $\psi : S_Y \to \mathbb R$ such that
\begin{equation} f(x,y) = \varphi (x) + h(x) (y-x) - \psi(y), \quad (x,y) \in S. \label{WEPf}\end{equation}
Moreover, let $f:S \to \mathbb R$ be an arbitrary function. We say that WEP($f$) holds in $S$ if (\ref{WEPf}) is satisfied for suitable functions $h,\varphi, \psi$.

We say that the WEP holds for $Q \in \mathcal M(\mu,\nu)$  if the WEP holds for the support $S$ of $Q$.
\end{definition}

\begin{example}[Binomial mesh]
Let $x, y_1, y_2$ be some given positive numbers with $y_1 \neq y_2$ and let $S:= \{(x,y_i): i=1,2\}$ be the corresponding 2-mesh. Then the WEP holds for $S$.
Indeed for any given $f$, set $\psi=0$, $h(x)=\frac{f(x,y_1)-f(x,y_2)}{y_1-y_2}$, and $\varphi(x) = \frac{x-y_2}{y_1-y_2}f(x,y_1)+\frac{y_1-x}{y_1-y_2}f(x,y_2)$.
\end{example}

\begin{remark} \emph{Notice that by defining the WEP as a property of a given set of paths $S$, 
we possibly include sets that cannot be supports of martingale probability measures, as seen on the previous example
when, for instance, $y_1 > y_2 > x >0$. More generally, if the WEP holds for some set $S \subset X \times Y$ and if $m: \mathbb R_+ \to \mathbb R_+$ is an application such that
$id+m$ is invertible, where $id$ stands for the identity map, then the WEP holds for the set $S_m = \{(x+m(x), y): (x,y) \in S \}$ as well. 
Indeed, consider a function $f : S_m \to \mathbb R$ and let $(\varphi, h, \psi)$ be the decomposition of the function $g(x,y)=f(x+m(x), y)$.
We have $f(x+m(x), y) =\varphi(x)+h(x)(y-x)-\psi(y)=\varphi(x)+h(x)m(x)+h(x)(y-(x+m(x)))-\psi(y)$ so that by setting
$$\varphi_f(x+m(x))=\varphi(x)+h(x)m(x), \quad h_f(x+m(x))=h(x),$$
the WEP for $S_m$ follows since $id+m$ is invertible. This shows in particular that the WEP is a property of purely geometric and combinatorial nature.}\end{remark}

We conclude this section with showing that the WEP is a sufficient condition for the extremality of a measure $Q$ in $\mathcal M(\mu,\nu)$ when the support of $\nu$ is essentially ``generated'' by finitely many points in the support of $\mu$ (see the statement of the following proposition). This happens, for instance, when either the support of $\mu$ or the one of $\nu$ are finite.  This is not surprising since the WEP is {\it exact} on the support of $Q$.

In those situations, one can get for free the integrability of the terms in the WEP decomposition, and therefore the extremality. 

\begin{proposition}\label{WEP-extr}
Let $Q \in \mathcal M(\mu,\nu)$ and let $S$ be its support.  Assume that 
there is a finite set $\{x_1,\ldots ,x_n\} \subset S_X$ such that $S_Y = \cup_{1 \leq i \leq n} Y(x_i)$ (for instance, if either $S_X$ or $S_Y$ is finite). If the WEP holds for $Q$, then $Q$ is extremal.
\end{proposition}

\begin{proof}
Take some function $g \in L^1(Q)$, so that in particular $\int | g(x,y)| q(x ,dy) < \infty $ for all $x$ in the support of $\mu$, hence for all $x \in \{x_1,\ldots,x_n\}$.
By assumption there exist measurable functions $\varphi,h:S_X \to \mathbb R$ and $\psi:S_Y \to \mathbb R$ such that 
\begin{equation} \label{wep-g} g(x,y)=\varphi(x)-\psi(y)+h(x)(y-x) \end{equation} 
on the support of $Q$. With a slight abuse of notation, we still indicate by $\varphi, \psi$ and $h$ the extension of such functions to the whole respective spaces, i.e. $X$ for $\varphi, h$ and $Y$ for $\psi$, by setting them equal to $0$ outside $S_X$ for $\varphi,h$ and outside $S_Y$ for $\psi$. 

Now, for every $y \in S_Y$, $|\psi(y)| \leq |\varphi(x)| +|h(x)|(x+y) + |g(x,y)|$ for some $x \in \{x_1,\ldots ,x_n\}$ which entails 
\[ |\psi(y)| \leq \max_{1\le i \le n} |\varphi(x_i)| +\max_{1\le i \le n}(|h(x_i)|x_i)+\max_{1\le i \le n} |h(x_i)| y + \max_{1\le i \le n} | g(x_i, y)|.\] Therefore, since $\max_{1\le i \le n} |g(x_i,y)| \le \sum_{i=1}^n |g(x_i,y)|$ and each function $y \mapsto g(x_i,y)$ is $q(x_i,dy)$-integrable for all $i=1,\ldots,n$, we have
\[ \int | g(x_i,y)| d\nu (y) = \int | g(x_i,y)| q(x_i ,dy) \mu(\{x_i\}) < \infty.\]
Moreover, the function $y \mapsto \max_{1\le i \le n} |h(x_i)| y$ also belongs to $L^1(\nu)$, yielding that $\psi \in L^1(\nu)$ as well. After rearranging the terms and taking conditional expectations in \eqref{wep-g} we get 
\[ \varphi(x) = \int g(x, y) q(x,dy) + \int \psi(y) q(x,dy) \] 
so that $|\varphi(x)| \leq \int (| g(x,y)| + |\psi(y)|) q(x,dy)$. Since $g \in L^1(Q)$ and $\psi \in L^1(\nu)$, we also have that $\varphi \in L^1(\mu)$. Then by difference $(x,y) \mapsto h(x)(y-x)$ belongs to $L^1(Q)$ as well.

Hence we have shown that for every integrable function $g \in L^1 (Q)$, each term in its WEP decomposition \eqref{wep-g} is integrable for the respective measures, therefore by applying Theorem \ref{douglas2} we get that $Q$ is extremal.
\end{proof}

\begin{remark}
\emph{When both $\mu$ and $\nu$ have finite support, WEP and extremality are actually equivalent. To see this, just notice that $L^1(Q)$ can be identified with the set of all functions $f : \supp(Q) \to \mathbb R$, which is a finite dimensional vector space. Hence any dense subspace of $L^1(Q)$ equals $L^1(Q)$ so that the WEP and the weak PRP coincide. In particular,
every extremal measure $Q$ satisfies the WEP.}
\end{remark}

\begin{remark}
\emph{In general, the set of semi-static trading strategies is not closed in $L^1(Q)$, as it is showed in the article \cite{ALS}. More precisely, in their Theorem 1.1, the authors construct a discrete-time model, defined on a countable sample space, and a sequence of semi-static strategies converging in $L^p$, for every $p \geq 1$, to some limit which cannot even be dominated by the final outcome of a semi-static strategy. The problem of whether this would hold even if $Q$ is an extremal measure in $\mathcal M(\mu,\nu)$ is still open.}
\end{remark}

%%%%%%%%%%%%%%%%%%%%%%%%
\section{Some sufficient conditions for the WEP}\label{suff}
%%%%%%%%%%%%%%%%%%%%%%%%
In this section we provide some easily verifiable sufficient conditions for the WEP to hold.
\subsection{Intersection lemma}
We start with some preliminary results showing that the WEP (as well as extremality) entails quite a strong constraint on the intersection of the image sets $Y(x)$, for $x \in X$.

\begin{lemma} [Intersection Lemma under the WEP] \label{2points}
Assume that the WEP holds for $S$ and let $x_1, x_2$ be two distinct points in $S_X$. Then
  \begin{equation}\label{ineq2points} | Y(x_1) \cap Y(x_2) | \le 2.\end{equation}
\end{lemma} 
  
\begin{proof} Assume that there exist two distinct points $x_1,x_2 \in S_X$ such that $ Y(x_1) \cap Y(x_2) \supset \{ y_1,y_2,y_3\}$ with $y_1 < y_2 < y_3$. Hence, one can choose a function \(f: S \to \mathbb R \) such that \(f(x_1,\cdot)\) and $f(x_2,\cdot)$ have, respectively, strictly increasing and strictly decreasing increment ratios over the set $\{y_i$, $i=1,2,3\}$, i.e. 
\begin{equation}\label{increment1} \frac{f(x_1,y_2) - f(x_1,y_1)}{y_2 - y_1} < \frac{f(x_1,y_3) - f(x_1,y_2)}{y_3 - y_2} \end{equation}
and
\begin{equation}\label{increment2} \frac{f(x_2,y_2) - f(x_2,y_1)}{y_2 - y_1} > \frac{f(x_2,y_3) - f(x_2,y_2)}{y_3 - y_2}.\end{equation}
Since the WEP holds for $S$, $f$ can be represented as
\[ f(x,y) = \varphi (x) + h(x) (y-x) - \psi (y), \quad x \in S_X, \quad y \in Y(x),\]
for some functions $h,\varphi : S_X \to \mathbb R$ and $\psi : S_Y \to \mathbb R$. Hence (\ref{increment1}) and (\ref{increment2}) become
\[ \frac{\psi(y_2) - \psi(y_1)}{y_2 - y_1} < \frac{\psi(y_3) - \psi(y_2)}{y_3 - y_2}, \quad \frac{\psi(y_2) - \psi(y_1)}{y_2 - y_1} > \frac{\psi(y_3) - \psi(y_2)}{y_3 - y_2}.\]
leading to a contradiction.
\end{proof}

Even though we do not know in general the relationship between WEP and extremality, we can show that in the intersection lemma above the WEP can be replaced by the extremality property while keeping the same conclusion.

\begin{lemma} [Intersection Lemma under extremality] \label{inter-extr}
Assume that $Q$ is extremal. Let $S$ be the support of $Q$ and $x_1, x_2$ two distinct points in $S_X$. Then
  \begin{equation}\label{ineq2pointsExtr} | Y(x_1) \cap Y(x_2) | \le 2.\end{equation}
\end{lemma} 
  
\begin{proof} 
Let us proceed by contradiction and assume that there are at least three distinct points $\{y_1, y_2, y_3\}$ in  $Y(x_1) \cap Y(x_2) $. We are going to build a perturbation
of $Q$ in $\mathcal M(\mu,\nu)$, showing that $Q$ cannot be extremal. Consider the positive number 
\[ c := \inf \{Q(x_i, y_j): i = 1,2 , j =1,2,3 \} > 0.\]
Let us start with the path $(x_1, y_1)$ and perturbate its probability by some number $\alpha$, so getting a new probability weight $Q_1 (x_1,y_1)= Q(x_1,y_1)+\alpha$. Consider now the path $(x_2, y_1)$ and perturbate it, in order
to preserve the total mass at $y_1$, by $-\alpha$, i.e. define $Q_1(x_2,y_1)= Q (x_2,y_1) -\alpha$. In the same way we associate to the path $(x_2,y_2)$ a perturbation $\beta$,
to the path $(x_1, y_2)$ the opposite perturbation $-\beta$, and lastly $\gamma$ to the path $(x_1, y_3)$ and $-\gamma$ to the path $(x_2, y_3)$.
By choosing $\alpha, \beta, \gamma$ small enough, this procedure leads to a new probability measure $Q_1$ on $S$. By construction the total mass at each point $y_j$, for $j = 1,2,3$, is preserved.

On the other hand, the mass $\mu(\{x_1\})$ at the point $x_1$ is preserved if and only if $\alpha-\beta+\gamma=0$,
and the conservation of the mass at $x_2$, i.e. $\mu(\{x_2\})$, gives the same condition. It remains to check the martingale property, giving the same condition for both points $x_1$ and $x_2$, which is $\alpha y_1  -\beta y_2 + \gamma y_3=0$. Solving this system gives $\alpha = \beta - \gamma$
and $\beta (y_1-y_2) = \gamma (y_1-y_3)$. We can choose $\gamma$ sufficiently small so that $\max\{ |\alpha| , |\beta| , |\gamma|\} <c$, which guarantees that the perturbation $Q_1$ constructed above belongs to $\mathcal M(\mu,\nu)$. 

Finally, performing an analogue perturbation whose weights have the same absolute value but opposite signs than those in $Q_1$, we obtain another measure $Q_2$ in $\mathcal M(\mu,\nu)$ such that $Q=(Q_1+Q_2)/2$, which contradicts the extremality of $Q$.  
\end{proof}

\subsection{The 2-link property}
%%%%%%%%%%%%%%%

In this section we introduce the notion of ``2-link property'', which gives a sufficient condition for the WEP to hold on a given subset $S$ of $X \times Y$. 
This property can be viewed as a strengthening of the necessary condition given in the intersection Lemma \ref{2points}. 
It is simple to formulate and at the same time it gives an easy method to generate quite a rich family of supports of extremal measures.

\begin{definition}[2-link property]
We say that $S$ has the  \emph{$2$-link property} if \emph{there exists a numbering} $S_X=(x_n)_{n\ge 1}$ such that for all $n\ge 1$ we have
\begin{equation}\label{2LP}
 | Y(x_n) \cap \bigcup_{i=1}^{n-1} Y(x_i) | \le 2, \tag{2LP}
\end{equation}
with the convention $\bigcup_{i=1}^0 = \emptyset$.
\end{definition}

With a slight abuse of terminology we will sometimes use (\ref{2LP}) with the meaning of ``2-link property''. It will be clear from the context.

\begin{proposition}\label{rec-wep}
If $S$ has the  \emph{$2$-link property}, then the WEP holds for $S$.
\end{proposition}

\begin{proof}
Let $f: S\to \mathbb R$ be any measurable real-valued function, we want to find functions $\varphi, h : S_X \to \mathbb R$ and $\psi : S_Y \to \mathbb R$ such that $f(x,y) +\psi(y) =  \varphi(x) + h(x)(y-x)$ on $S$. We construct such functions by induction using the condition (\ref{2LP}), which we assume to be satisfied for at least one numbering $S_X=(x_n)_{n \ge 1}$. Take the first element $x_1$ and consider all the points $y \in Y(x_1)$. Pick arbitrarily two such points, say $y _1, y _2 \in Y(x_1)$, take any two real numbers $\psi_1, \psi_2$ and set $\psi(y_1):=\psi_1$ and $\psi (y_2):=\psi_2$. We want to have
\[ f(x_1, y_i) + \psi(y_i) = \varphi(x_1) + h(x_1)(y_i - x_1), \quad i=1,2,\]
so that the parameters $\varphi(x_1),h(x_1)$ of the affine function in $y \in Y$, $y \mapsto \varphi(x_1) + h(x_1)(y-x_1)$, are determined by the two points on the LHS in the equality above. As a consequence, the other values of the function $\psi(y)$ for $y \in Y(x_1) \setminus \{y_1,y_2\}$ are also determined via the equality
\[ \psi(y) = f(x_1, y) -  \varphi(x_1) - h(x_1)(y - x_1).\]
Now, assume that we have constructed the functions $h,\varphi : (x_i)_{i=1}^{n-1} \to \mathbb R$ and $\psi: \cup_{i=1}^{n-1} Y(x_i) \to \mathbb R$ such that
\[ f(x_i ,y) = \varphi(x_i) + h(x_i) (y-x_i) -\psi(y), \quad y\in Y(x_i), \quad 1 \le i \le n-1 .   \]
Consider the next point $x_n$ in the given numbering of $S_X$ satisfying the condition (\ref{2LP}). The latter implies that there exist at most two distinct points, say, $y_1,y_2 \in \cup_{i=1}^{n-1} Y(x_i)$ such that $y_i \in Y(x_n)$ for $i=1,2$. Let us consider first the case where these points are exactly two. Thus, the two values $\psi(y_i)$, $i=1,2$, have already been determined and so are $f(x_n,y_i)+\psi(y_i)$, $i=1,2$. This identifies completely and without any ambiguity the parameters $\varphi(x_n)$, $h(x_n)$ in the $y$-affine part of the following WEP representation
\[   f(x_n, y) + \psi(y) = \varphi(x_n) + h(x_n)(y - x_n), \quad y \in Y(x_n).\]
Indeed we have
\begin{eqnarray*}\label{hphi} h(x_n) &=& \frac{f(x_n,y_2) - f(x_n,y_1) + \psi(y_2) - \psi(y_1)}{y_2 - y_1},\\
\label{phi} \varphi(x_n) &=& f(x_n, y_1) - h(x_n)(y_1 - x_n) + \psi(y_1),\\
\label{psi} \psi(y) &=& f(x_n, y) -\varphi(x_n) - h(x_n) (y-x_n), \quad y \in Y(x_n)/\{y_1,y_2\}.
\end{eqnarray*}
Doing so, we have extended the functions $h,\varphi$ to the finite set $(x_i)_{i=1}^n$ and the function $\psi$ to the set $\cup_{i=1}^n Y(x_i)$. To complete this part we need to consider also the cases when the intersection $(\cup_{i=1}^{n-1} Y(x_i)) \cap Y_n$ is empty or contains only one point, say $y_1$. In the latter case, the construction is similar with the only difference that, while $\psi(y_1)$ is fixed, the value $\psi(y_2)$ can be chosen arbitrarily. In the former case, i.e. the intersection is empty, one can proceed as at the beginning of this proof.

Finally, by the induction principle, we can conclude that there exist functions $h, \varphi : S_X \to \mathbb R$ and $\psi: S_Y = \cup_{n\ge 1} Y(x_n) \to \mathbb R$ such that WEP($f$) holds for any arbitrary function $f$.
\end{proof}

\begin{remark}
{\rm It is easy to construct conditional supports with infinitely many points satisfying (\ref{2LP}). This is in contrast to what happens for martingale measures without 
constraints on the marginals, where the only extremal points have a two-point conditional support (compare to Lemme A in Dellacherie \cite{dell} and to Theorem 6 in Jacod and Shiryaev \cite{JS}). 
From a financial perspective, this is clearly due to the fact that the extremal points in $\mathcal M(\mu,\nu)$ correspond to semi-statically complete models as in Theorem \ref{douglas2}, where in particular one 
is allowed to trade statically in infinitely many European options.} 
\end{remark}

\begin{example}[``Binomial tree'']
{\rm Any probability $Q \in \mathcal M(\mu,\nu)$ whose conditional supports have two points, i.e. $| Y(x) | =2$ for all $x \in X$, is extremal. Indeed, property (\ref{2LP}) is trivially satisfied. } 
\end{example}

\begin{example}[Hobson and Klimmek \cite{HK} ``Trinomial tree''] \label{exHK}
{\rm Hobson and Klimmek \cite{HK} constructed an optimal martingale optimal transport whose conditional support is fully characterized as follows: there exist $a<b$ and two decreasing maps $p$ and $q$ such that for those $x$ with $x\in Y(x)$ we have \begin{itemize}
\item $Y(x) =\{x\}$ if $x<a$ or $x>b$,
\item $Y(x) = \{ p(x),x,q(x)\}$ otherwise, with $p(x)<a$ and $q(x) >b$.
\end{itemize}
Moreover for those $x$ which do not belong to $Y(x)$ we have $Y(x) =\{p(x),q(x)\}$. One can see that the property (\ref{2LP}) is satisfied in this case as well. Indeed, let $X$ be the (countable) support of $\mu$ and $X=X_{<a} \cup X_{[a,b]} \cup X_{>b}$, with $X_{>a} := \{ x \in X : x >a\}$ (the other two sets are defined analogously). Consider any numbering for those three sets, i.e. $X_{>a} = (x_n ^a)$, $X_{>b} = (x_n ^b)$ and $X_{[a,b]} = (\bar x_n)$. Hence by alternating elements of each sequence we get a numbering for $X$, given by $(x_n) = (x_1^a , x_1 ^b, \bar x_1, \ldots)$ which satisfies (\ref{2LP}). 
Note that Hobson and Klimmek optimal coupling with $\mu \wedge \nu =0$ is a binomial tree. More on this support can be found in Example \ref{exHK2}.}
\end{example}

In the paper \cite{BJ} Beiglb\"ock and Juillet introduce the fundamental notion of left-monotone martingale transport plan (see Definition 1.4 therein) as follows: a martingale transport plan $\pi$ on $\mathbb R \times \mathbb R$ is called \emph{left-monotone} if there exists a Borel set $\Gamma \subset \mathbb R \times \mathbb R$ with $\pi(\Gamma)=1$ and such that whenever $(x,y^- ),(x,y^+ ),(x',y') \in \Gamma$ we cannot have 
\begin{equation} x<x' \quad \textrm{and} \quad y^- < y' < y^+ . \label{left-mon} \end{equation} In Theorem 5.1 in \cite{BJ}, it is proved that there exists a unique left-monotone transport plan in $\mathcal M(\mu,\nu)$, which is denoted by $\pi_{\textrm{lc}}$ and called \emph{left curtain}.
The right curtain $\pi_{\textrm{rc}}$  is defined similarly just by replacing (\ref{left-mon}) with the following forbidden pattern: whenever $(x,y^- ),(x,y^+ ),(x',y') \in \Gamma$ we cannot have \begin{equation} x>x' \quad \textrm{and} \quad y^- < y' < y^+ . \label{right-mon} \end{equation}

\begin{proposition} Assume that there exists a strictly decreasing (resp. strictly increasing) numbering for $S_X$, i.e. $S_X=(x_n)_{n \ge 1}$ with $x_1 > x_2 > \cdots$ (resp. $x_1 < x_2 < \cdots$). Then, the left 
(resp. right) curtain $\pi_{\textrm{lc}}$ (resp. $\pi_{\textrm{rc}}$) satisfies the property (\ref{2LP}). In particular, it satisfies the WEP and so, under the assumptions in Proposition \ref{WEP-extr},
it is extremal in $\mathcal M(\mu,\nu)$. \end{proposition} 

\begin{proof} Assume by contradiction that all numberings of $S_X$ do not satisfy (\ref{2LP}), hence the decreasing order $x_1 > x_2 > \cdots $ in particular does not fulfil it. Therefore, there exists $k \ge 2$ such that $Y(x_k) \cap (\cup_{1\le i\le k-1} Y(x_i))$ contains three distinct points $y',y'',y'''$ in $Y$. We can order them as $y^u > y^m >y^l$. There exist $x_i$ with $i=1,\ldots, k-1$ such that $y^m \in Y(x_i)$. Then, we have found that $(x_k,y^j)$, with $j\in \{u,m,l\}$, belongs to the support of the left curtain together with $(x_i,y_m)$, where we recall that $x_i > x_k$. This is exactly the forbidden mapping (\ref{left-mon}) in the left curtain definition (see also \cite[Figure 1]{BJ}). Hence, the left curtain support satisfies (\ref{2LP}) and it satisfies the WEP (cf. Proposition \ref{rec-wep}). As a consequence, under the assumptions in Proposition \ref{WEP-extr}, the left curtain is extremal in $\mathcal M(\mu,\nu)$. 
The proof for the right curtain is similar.
\end{proof}

Now we provide an example showing that the $2$-link property is not necessary for the WEP.

\begin{example} \label{patras}
{\rm The picture below shows a subset $S$ of $X \times Y$, with $X = \{x_i\}_{i=1}^4$ and $Y=\{y_j\}_{j=1}^6$, which does not have the two-link property 
and nonetheless one can check by direct verification that the WEP is fulfilled (see Example \ref{patrasWEP}).

\begin{figure}[H]
\centering
\includegraphics[scale=1]{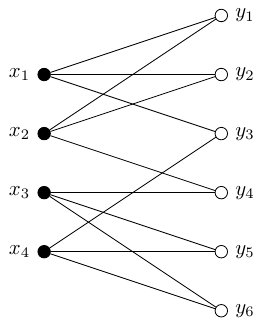}
\end{figure}
\noindent We will come back to this picture in the next section (see Example \ref{patras2}).}
\end{example}

\begin{example}\label{infinite-2LP}
{\rm We conclude with an example of an infinite support fulfilling \eqref{2LP}. Consider $S_X = (x_n)_{n \ge 1}$ for some given sequence of positive numbers such that $|Y(x_n)|=n$, for all $n \ge 1$, and satisfying the properties
\[ |Y(x_2) \cap Y(x_1)| =1, \quad \left |Y(x_{n+1}) \cap \bigcup_{i=1} ^n Y(x_i)\right | = 2, \quad n \ge 2.\]
Clearly with this construction we have that $S_Y = \cup_{n \ge 1} Y(x_n) $ is infinite as well. One possible picture of the first four iteration steps is the following

\begin{figure}[H]
\centering
\includegraphics[scale=1]{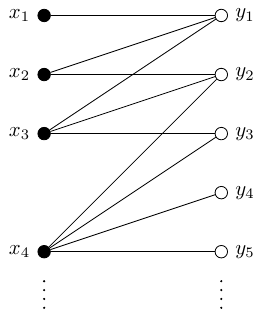}
\end{figure}
This is just one possible example with infinitely many points on both sides of the support of many more that can be provided using the definition of (2LP). Indeed, we stress once more that (2LP) is a very constructive property by its own nature.}
\end{example}

\subsubsection{Relation to graph theory}
%%%%%%%%%%%%%%%%%%%%%
We conclude this section by showing that the 2-link property is very much related to the notion of $k$-degeneracy in graph theory as in, e.g., \cite{lick}. In particular, we see how this unexpected relation could provide alternative characterizations of (\ref{2LP}) as well as a way to generate subsets of $X \times Y$ satisfying (\ref{2LP}). We will use a little terminology of graph theory, for whom we refer to Diestel's book \cite{diestel}. 

Let $G=(V,E)$ be a graph with vertex set $V$ and edge set $E$. For our purposes, we allow $G$ to have multiple edges. Moreover, $V$ and $E$ can be infinite countable sets. The degree $d(v)$ of a vertex $v\in V$ is the number of edges incident with $v$. The smallest degree among the vertices of $G$ is called \emph{minimum degree} of $G$ and is denoted by $\delta(G)$.
Moreover, a subgraph $H$ of a graph $G$ consists of a subset of the vertices of $G$ and a subset of the edges of $G$ which together form a graph. The \emph{subgraph induced} by a set $U$ of vertices of $G$, denoted by $\langle U \rangle$, has $U$ as its vertex set and contains all edges of $G$ incident with two vertices of $U$.
\begin{definition}[\cite{lick}]
A graph $G$ is said to be $k$-degenerate, for $k$ a nonnegative integer, if for each induced subgraph $H$ of $G$, we have $\delta (H) \le k$.
\end{definition}

The following proposition states the equivalence between $k$-degeneracy of a graph, with $k=2$, and a property very similar to (\ref{2LP}). When $G$ is finite and does not have multiple edges, this is just Proposition 1 in \cite{lick}. However, even when $G$ can have countably many vertices and multiple edges, such an equivalence still holds. We provide the proof in the case $k=2$ for reader's convenience.

\begin{proposition}[\cite{lick}] \label{2dege}
$G=(V,E)$ is $2$-degenerate if and only if the set of its vertices $V$ admits an order $V=(v_n)_{n \ge 1}$ such that $d(v_1) \le 2$ and, in the induced subgraph $\langle \{v_1,\ldots, v_{n-1}\}\rangle$ of $G$, we have $d(v_n) \le 2$, for each $n \ge 1$.
\end{proposition}

\begin{proof}
Assume that $G$ is 2-degenerate. One can find such an ordering as follows: pick the vertex with the smallest degree, name it $x_1$ and remove it from the graph. Repeat the procedure with the remaining subgraph and iterate. Now, assume that the order is given and there exists an induced subgraph $H$ of $G$ with $\delta (H) > 2$. Choose $n$ sufficiently large so that $V(H)/\{v\} \subset \{v_1, \ldots, v_{n-1}\}$, where $V(H)$ is the vertex set of $H$ and $v$ is one of its vertices. Now, since $\delta (H) > 2$, the degree of $v$ in the induced subgraph $\langle v_1, \ldots, v_{n-1}\rangle$ is strictly bigger than $2$, which contradicts the property of the ordering. 
\end{proof}

We show how 2-degenerate graphs can be used to generate subsets $S \subset X \times Y$ fulfilling the 2-link property. First, notice that $S$ can be viewed as a (possibly infinite) bipartite undirected graph $G=(V,E)$, where $V = X \times Y$ is the set of vertices and $E=S$ is the set of edges so that $e=xy$ is an edge of $G$ if and only if $(x,y) \in S$, i.e. $y \in Y(x)$. For our purpose, let us define a simpler graph with vertex set $X$. Let $G^X = (V^X, E^X)$ be a graph with $V^X = X$ and $E^X$ is such that $x_1x_2 \in E^X$ if and only if $x_1 y$ and $x_2 y$ belong to $E$ for some $y \in Y$, \emph{with the constraint that the same $y$ cannot be used more than twice}. Notice that the graph $G^X$ can have multiple edges. Moreover, different $G^X$ can be constructed starting from the same graph $G$. Let us illustrate this construction in the following example: let $G$ be as in the picture below:

\begin{figure}[H]	
\centering
\includegraphics[scale=1]{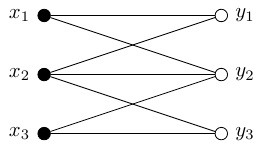}
\end{figure}

Notice that $G$ satisfies (\ref{2LP}). One possible graph $G^X$ produced out of $G$ as described above is given by
\begin{figure}[H]	
\centering
\includegraphics[scale=1]{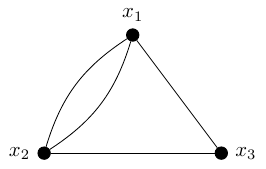}
\end{figure}

\noindent  the others can obtained from the latter by re-labeling the vertices, in other terms such graphs are isomorphic.\footnote{Let $G=(V,E)$ and $G'=(V',E')$ be two graphs. We call $G$ and $G'$ isomorphic if there exists a bijection $\eta : V \to V'$ such that $xy \in E$ if and only if $\eta(x)\eta(y) \in E'$ for all $x,y \in V$. Such a map is called isomorphism (cf. \cite[Sec. 1.1]{diestel}).} Let us denote $\widetilde G^X$ the associated equivalence class. The following equivalence is an immediate consequence of Proposition \ref{2dege} and the way $G^X$ has been defined. 
\begin{proposition}
Let $S$ be a subset of $X\times Y$ and let $G$ be the corresponding graph. If $\widetilde G^X$ is 2-degenerate, then $S$ (or equivalently $G$) satisfies (\ref{2LP}).
\end{proposition}

The article \cite{lick} contains many examples of $k$-degenerate (finite) graphs, that could be used to generate supports of extremal measures in $\mathcal M(\mu,\nu)$. A relevant class of examples is the class of connected 3-regular graphs (all vertices have exactly three neighbours). Such graphs are not 2-degenerate themselves, but deleting any of their vertices leaves a 2-degenerate graph. A way to generate graphs fulfilling (\ref{2LP}) is the following: given a 2-degenerate graph $G^X = (V^X, E^X)$ with vertex set $V^X = X$, we define a graph $G=(V,E)$ with vertex set $V=X \times Y$ and edge set $E$ obtained by splitting any edge $e=x x' \in E^X$ into two edges $xy$ and $x'y$ for some $y \in Y$. The new graph $G$ satisfies (\ref{2LP}) by construction. 

A deeper understanding of (\ref{2LP}) and 2-degeneracy within graph theory goes beyond the scope of this paper. We postpone the study of this interplay for future research.

\subsection{Erasable sets}
%%%%%%%%%%%%%%

In this section we provide another sufficient condition for the WEP. While the 2-link property imposes a kind of compatibility condition among 
paths when adding more and more points, the condition given here is based on erasing paths in a certain way. This is motivated by 
the analogous property in the non martingale case (cf \cite{muk}, Theorem 2.3). We start with the following lemma:

\begin{lemma}\label{reduction}
Let $S \subset X \times Y$ and let $U = \{(x,y) \in S : |X(y)|=1 \text{\emph{ or }} |Y(x)|\leq2\}$. Then the WEP holds for $S$ if and only if 
it holds for $S \setminus U$. \end{lemma}

\begin{proof}
It is clear that if the WEP holds for $S$, then it holds for $S\setminus U$. Now, assume that the WEP holds for $S \setminus U$. Let $(x,y) \in U$. Consider first the case when $|X(y)|=1 $ and $x \in X(y^\prime)$ for some $y^\prime$ such that $(x,y^\prime) \in S \setminus U$. In this case the value $\psi(y)$ can be taken as $\psi(y) := f(x,y)-\varphi(x)-h(x)(y-x)$ where $\varphi(x)$ and $h(x)$ are given
by the WEP for  $S \setminus U$. On the other hand, if $|Y(x)|=2 $ with $Y(x) = \{y_1,y_2\}$, then $\varphi(x)$ and $h(x)$ are uniquely determined by the values of $\psi$ on $Y(x)$ via the following equations:
\begin{equation} h(x) = \frac{\psi(y_1) - \psi(y_2)}{y_1 - y_2}, \quad \varphi(x) = \frac{x-y_2}{y_1-y_2}\psi(y_1) + \frac{y_1 - x}{y_1 - y_2}\psi(y_2). \end{equation}
If $|Y(x)|=1$, i.e. $Y(x) = \{y\}$, then the choice $\varphi(x)=\psi(y)+f(x,y)$ and $h(x)=0$ allows to satisfy the WEP along the path $(x, y)$.
\end{proof}
Let us define for a subset $S \subset X \times Y$ the following {\it erasure} transformations:
\begin{eqnarray*}
\mathcal E_{1,x}(S) &=& S \setminus \{(x,y) \in S : |X(y)|=1 \},\\
\mathcal E_{k,y}(S) &=&  S \setminus \{(x,y) \in S :  |Y(x)|=k\}, \quad k=1,2,
\end{eqnarray*}
and finally $\E = \E_{2,y} \circ \E_{1,y} \circ \E_{1,x}$.

\begin{definition} [Erased sets and fully erasable sets] A set $S$ is called \emph{erased} if $\E(S)=S$. Moreover, it is called \emph{fully erasable} if $\E^n (S) \downarrow \emptyset$ as $n \to \infty$, i.e. for all $(x,y) \in S$ there exists $n \in \mathbb N$ such that $(x,y) \notin \E^n(S)$ (where by convention we set $\E^0 = id$).
\end{definition}

\begin{example}[Hobson and Klimmek \cite{HK} ``Trinomial tree'' (cont'd)] \label{exHK2}
{\rm Hobson and Klimmek trinomial tree (as defined in Example \ref{exHK}) is fully erasable: indeed $|Y(x)|=1$ if $x \leq a$ or $x \geq b$. For $x\in (a,b)$, it follows from
the definition of the transition probabilities that $|X(x)|=1$. Since the meshes originating from $x\in (a,b)$ are trinomial meshes, after applying the maps $\mathcal E_{1,x}$ and $\mathcal E_{1,y}$ we are therefore left with binomial meshes,
and eventually $\E(S)= \emptyset$.
}
\end{example}
\begin{example}
{\rm This is a non-trivial example of a fully erasable set $S$:
\begin{figure}[H]	
\centering
\includegraphics[scale=1]{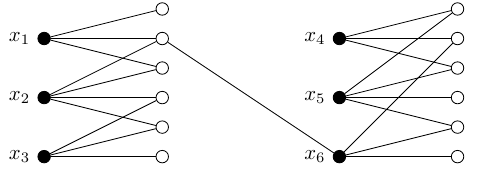}
\end{figure}
\noindent The support above can be erased, using the definition of fully erasable sets, along the following steps: first, applying $\mathcal E_{1,x}$ we get
 \begin{figure}[H]	
\centering
\includegraphics[scale=1]{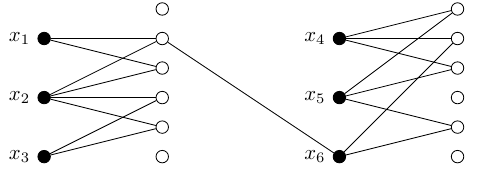}
\end{figure} \noindent
while using $\mathcal E_{1,y}$ and $\mathcal E_{2,y}$ gives
\begin{figure}[H]	
\centering
\includegraphics[scale=1]{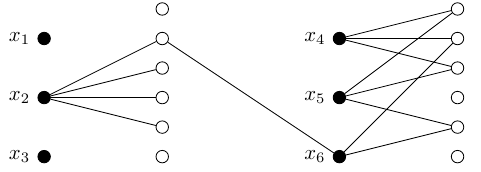}
\end{figure} \noindent
Moreover an immediate application of $\mathcal E_{1,x}$ and then $\mathcal E_{1,y}$ again, we obtain
 \begin{figure}[H]	
\centering
\includegraphics[scale=1]{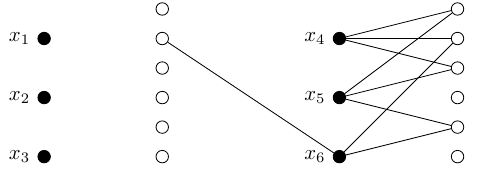}
\end{figure} \noindent
One last application of $\mathcal E_{2,y}$ and finally $\mathcal E_{1,x}$ erases the support fully.
}
\end{example}

\begin{example}[An infinite fully erasable support]
{\rm By induction, we are going to build first an auxiliary non-erasable infinite set, which will be slightly perturbated to obtain a fully erasable one. In a few words, the auxiliary set will have only trinomial meshes, its left-hand points will be ordered and every right-hand point will have exactly two paths coming to it. Moreover, it will be connected in the sense that any right- or left-hand point will be connected to the initial mesh in the iteration.

We start with a trinomial mesh $S_1=M(x_1)$ with $M(x_1)_Y = \{y_1 ^1, y_1 ^2 , y_1 ^3\}$ all distinct, to which we attach one more trinomial mesh $M(x_2)$ with right-hand points $\{ y_2 ^1, y_2 ^2 ,y_2 ^3\}$ chosen so that 
\[ y_2^1 = y_1 ^1, \quad y_2 ^2 = y_1 ^2, \quad \textrm{and } y_2 ^3 \neq y_1 ^i, \quad i=1,2,3.\] By induction we assume to have a set $S_n$ with the desired properties. Hence in order to continue the construction we consider a further trinomial mesh $M(x_{n+1})$ and set $S_{n+1}=S_n \cup M(x_{n+1})$. We need to specify how the new mesh is connected with the previous points. Let $M(x_{n+1}) = \{ y_{n+1}^1,y_{n+1}^2, y_{n+1}^3\}$ (three distinct points), the right-hand points in the previous set $S_n$ can be partitioned as $(S_n)_Y = F^1 _n \cup F^2 _n$ where we denote $F_n ^k := \{ y \in (S_n)_Y : |X(y) | =k\}$ for $k=1,2$. Let us assume that $F_n ^1 \neq \emptyset$. Hence we consider the following two cases: \begin{itemize}
\item $|F_n ^1|=1$, in which case we choose $y_{n+1} ^1 \in F_n ^1$ (it is the only available point), while we pick $y_{n+1}^j \notin (S_n)_Y$ for $j=2,3$.
\item $|F_n ^1| \ge 2$, in which case we choose $y^1_{n+1},y^2_{n+1} \in F_n ^1$ and $y_{n+1}^3 \notin (S_n)_Y$, hence after adding the new mesh we have $|X(y_{n+1} ^j)| =2$ for $j=1,2$.
\end{itemize}
Then the set $S_{n+1}$ has all the required properties and it satisfies $F_{n+1}^1 \neq \emptyset$. Set now $S := \cup_{n \ge 1} S_n$. It is readily checked that every left-hand point in $S$ belongs to exactly three paths, that every right-hand point belongs to exactly two paths, and by construction every path is connected to the initial mesh $M(x_1)$. In particular, we have that $S$ is not fully erasable. Moreover, the sequence $(x_n)_{n\ge 1}$ can be taken strictly increasing.

Pick now any right-hand point $y$ in $S_Y$, which belongs exactly to two paths, say $(z_1, y)$ and $(z_2,y)$, and replace them by $(z_1, y+\varepsilon)$ and $(z_2,y-\varepsilon)$ where $\varepsilon$ is such that neither $y+\varepsilon$ nor $y-\varepsilon$ belongs to $S_Y$. Let $S(\varepsilon)$ stand for the new set. Then one can show that it is fully erasable. Indeed take any path connecting $x_1$ to $y$. By induction from $y$, every sub-path $(x,z)$ can be erased, since it will either belong to a binomial mesh, or it will satisfy $|X(z)|=1$. So eventually $M(x_1)$ will be erased too. 

Choose now the smallest element in $S_X$ (well defined since $(x_n)_{n \ge 1}$ is strictly increasing) such that its mesh has not been erased yet. Since it was also connected to $x_1$ in $S$, there is a sub-path which will either belong to a binomial mesh, or it will satisfy $|X(z)|=1$ along the former path connecting it to $x_1$. So we can erase inductively from this sub-path. Applying iteratively such a scheme will fully erase the set $S(\varepsilon)$.}
\end{example}

The next result shows that the full erasability implies the WEP.

\begin{proposition} \label{prop:fullyWEP}
Assume $S$ is fully erasable. Then the WEP holds for $S$.
\end{proposition}

\begin{proof}
Notice that $S$ is fully erasable if and only if $S = \cup_{n\ge 0} \E^n (S) ^c$ (with the convention $\E^0 =id$). Hence to prove that 
the WEP holds for $S$ we can proceed by induction over $n$ as follows. First $\E^0 (S)^c = \emptyset$, so the WEP holds for $n=0$. Assume that 
the WEP holds for $\E^n (S)^c$ and let us prove that it holds for $\E^{n+1}(S) ^c $ as well. By definition of erasure transformations, any 
$(x,y) \in \E^{n+1}(S)^c \setminus \E^n (S)^c = \E^n (S) \setminus \E^{n+1}(S)$ satisfies $|X(y)| =1$ or $|Y(x)| \in \{1,2\}$. Hence to extend the WEP 
from $\E^n (S)^c$ to the path $(x,y)$, we can proceed as in the proof of Lemma \ref{reduction}. 
\end{proof}

\begin{remark} \emph{Being fully erasable is not a necessary condition for the WEP: one can directly check that the set in Example \ref{patras} is not fully erasable. On the other hand, it can be proved that it is necessary in some special cases (see Proposition \ref{5-erasability}).}\end{remark}

Note that removing a path $(x,y)$ such that $|X(y)|=1$ may prevent the remaining set from being the support of a martingale measure. Indeed, consider the set $\{ (x,y_i): i=1,2\}$ 
with $0<y_1<x<y_2$.  We shall need in the sequel the following weakenings of the notion of erased set, which do preserve the martingale property:

\begin{definition} [$1$-erased and $2$-erased sets] \label{2-erased}
A set $S$ is called \emph{$1$-erased} (resp. $2$-erased) if $\mathcal E_{1,y}(S) =S$ (resp. $\mathcal E_{2,y} \circ \mathcal E_{1,y}(S) =S$), or, equivalently, if $|Y_S(x)|\geq 2$ (resp. $|Y_S(x)|\geq 3$) for all $x \in S_X$.
\end{definition}

\begin{remark} \emph{According to the definition above, 1-erased or 2-erased sets may have right-hand free paths, i.e. paths $(x,y)$ such that $|X(y)|=1$, unlike erased sets $S$
for which any point in $S_X$ is connected through $S$ to at least three points in $Y$,
and any points in $S_Y$ is connected to at least two points in $X$ through $S$. This implies that erased sets have at least two points in their projection onto $X$, and three points
on their projections onto $Y$.} \end{remark}

We conclude this section by investigating the relation between the notions of fully erasability and the $2$-link property. It turns out that in the finite case they are equivalent, 
while when $X$ is infinite one can easily build an example of a support satisfying the latter and not the former.%\newpage

\begin{proposition}
Assume $S_X$ is finite. Then $S$ is fully erasable if and only if it has the $2$-link property.
\end{proposition}
\begin{proof}
Let $|X|=n$ for some nonnegative integer $n \geq 1$. Assume that $S$ satisfies the $2$-link property. Hence, $S$ can be constructed as the union of finitely many sets $(S_k)_{k=1}^n$ as follows: $(S_1)_X$ is a singleton and at each further step $S_k$ is obtained by adding to $(S_{k-1})_X$ a new point, say $x_k$, such that the property (\ref{2LP}) is fulfilled. Now, starting from the bottom of such a construction, notice that any pair $(x_n,y)$ can be erased by applying the transformation $\mathcal E$ since the pairs $(x_n ,y)$ with $|X(y)| =1$ will be cancelled first and then any other pairs $(x_n,y)$ will follow since, after the first cancellation, they would satisfy $|Y(x_n)| \le 2$. Iterating $\mathcal E$ will have the same effect on every other pairs $(x_k,y)$, $y \in Y(x_k)$, $1 \le k \le n-1$, of the support $S$, which will be reduced to the empty set. Hence $S$ is fully erasable. 

Now, assume that $S$ is fully erasable. Since $S_X$ is finite, $S$ is fully erasable if and only $\E ^n (S) = \emptyset$ for some $n \geq 1$. The empty set trivially satisfies the $2$-link property. Now, we can proceed by induction. We assume that $\E^k(S)$ satisfies the $2$-link property and we want to prove that $\E^{k-1}(S)$ does as well. By definition of the erasure transformation $\E = \E_{2,y} \circ \E_{1,y} \circ \E_{1,x}$, we have that $\E^k(S)$ has been obtained by erasing from $\E^{k-1}(S)$ some pair $(x,y) \in S$ in the following order: first those satisfying $|X(y)| =1$, second those with $|Y(x)| =1$ and finally those having $|Y(x)| =2$. The key observation is that adding them up to $\E^k (S)$ to go back to $\E^{k-1}(S)$ transfer the $2$-link property to the bigger set $\E^{k-1}(S)$. Hence, by the induction principle, we can conclude that $S=\E^0(S)$ satisfies the $2$-link property.\end{proof}

\begin{example}
{\rm Here we show how to construct a set $S \subset X \times Y$, with $X$ and $Y$ countable subsets of $\mathbb R_+$, which satisfies the $2$-link property and which is not fully erasable. We want the support $S$ to satisfy $|Y(x)| \geq 3$ and $|X(y)| \geq 2$ for all $(x,y) \in S$. We start from some $x_1 \in X$ with $Y(x_1) = \{y_{1,1},y_{1,2},y_{1,3}\}$. Then we add a second point $x_2 \in X \setminus \{x_1\}$ with two links with $x_1$ and one free $y$-point attached to it, i.e. $Y(x_2) = \{ y_{1,1},y_{1,2},y_{2,1}\}$ for some $y_{2,1} \in Y \setminus Y(x_1)$. We continue the construction in such a way that $X(y_{2,1})$ has at least two points in $X$. Hence, we add $x_3 \in X \setminus \{x_1,x_2\}$ with, for instance, $Y(x_3) = \{y_{1,3},y_{2,1},y_{3,1}\}$. So far, the $2$-link property is fulfilled by construction. Now, consider the left-hand free point of $x_3$, i.e. $y_{3,1}$, and add a fourth point $x_4$ such that $Y(x_4) = \{y_{3,1},y_{4,1},y_{4,2}\}$ and so on. By iterating we will eventually get a set $S$ with the required properties. The next picture illustrates the first four steps of the construction:}
\begin{figure}[H]	
\centering
\includegraphics[scale=1]{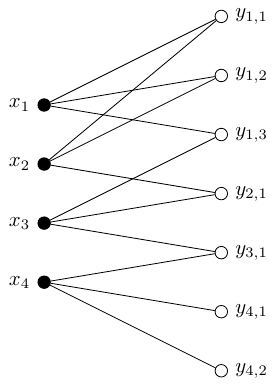}
\end{figure}
\end{example}

%%%%%%%%%%%%%%%%%%%%%%%%%
\section{A geometrical characterization of the WEP}\label{q-affine-functions-and-2-nets}
%%%%%%%%%%%%%%%%%%%%%%%%%

The goal of this section is to provide a characterization of sets $S \subset X \times Y$ satisfying the WEP. 
The main result is stated in Theorem \ref{thm2net}, which is based on the new notion of deadlocks (introduced in Definition \ref{deadlock}).

Now, we introduce the preliminary intuitive notion of connectedness in the following definition. Given a binary relation $R$ on $S$, we recall that the transitive closure of $R$ is defined as the smallest transitive relation over $S$ containing $R$. 

\begin{definition}\label{connected} Let $S \subset X \times Y$.
We say that two paths $(x,y)$ and $(x',y')$ in $S$  are \emph{neighbours} if either $x=x'$ or $y=y'$. The transitive closure of this relation is an equivalence relation. The corresponding equivalence
classes of $S$ will be called \emph{connected subclasses} (of $S$). A set $S$ with a single subclass will be called \emph{connected}.
\end {definition}

This notion of connectedness induces the following decomposition property, which will allow us to work with 1-erased connected sets without loss of generality.

\begin{proposition}\label{decompWEP}
Let $S \subset X \times Y$ be a 1-erased set and let $S=\cup_{n \ge 1} S_n$ be its decomposition into mutually disjoint connected subclasses. Then each $S_n$, for $n \ge 1$, is 1-erased.
Moreover, the WEP holds for $S$ if and only if it holds for each set $S_n$, $n \ge 1$.
\end{proposition}
\begin{proof} Assume that $S$ is 1-erased and that $S=\cup_n S_n$ is its decomposition into mutually disjoint connected subclasses. Assume by contradiction that $S_n$ is not 1-erased for some $n \ge 1$. 
Hence there exists $(x,y) \in S_n$ such that $\{y\}=Y_{S_n} (x)$. Since $S$ is 1-erased we have $|Y_S (x)| \ge 2$, hence there exists a point $y_0$ in $Y_S (x) \setminus Y_{S_n}(x)$. 
Moreover $(x,y_0)$ and $(x,y)$ are neighbours, which implies that $S_n$ cannot be a connected subclass of $S$. 

If the WEP holds for $S$ then it clearly holds for each subclass $S_n$ as well. Assume now that the WEP holds for every subclass $S_n$, $n\ge 1$. Hence every function $f: X \times Y \to \mathbb R$ satisfies the WEP over each subclass $S_n$, i.e.
\[ f(x,y)= \varphi_n (x) + h_n(x)(y-x) - \psi_n (y), \quad (x,y)\in S_n, \quad n\ge 1,\]
for some functions $\varphi_n , h_n : (S_n)_X \to \mathbb R$ and $\psi_n : (S_n)_Y \to \mathbb R$. Since the subclasses $S_n$ are disjoint, we can safely paste such functions together and get the WEP for $f$ over the whole set $S$.
\end{proof}

\subsection{$S$-affine functions and 2-nets: definitions and properties}
%%%%%%%%%%%%%%%%%%%%%%%%%%%%%%%%%

This sub-section and the next one will focus on the two important auxiliary notions of $S$-affine functions and 2-nets. From now on we will work under the following standing assumption:
\begin{assumption}
$S \subset X \times Y$ is a 1-erased set, i.e. $\mathcal E_{1,y}(S) =S$ or equivalently $|Y(x)| \ge 2$ for all $x \in S_X$.
\end{assumption}

\begin{definition}[S-affine functions]
A function $\psi : S_Y \to \mathbb R$ is called  \emph{$S$-affine} if it is affine on each set $Y_S(x)$, i.e., for all $x \in S_X$ there exist $\varphi(x), h(x)$ such that 
\begin{equation}\label{Qaff} \psi(y)=\varphi(x)+h(x)(y-x), \quad \forall y \in Y_S(x). \end{equation}
$\aff(S)$ denotes the set of all $S$-affine functions.
\end{definition}

Basically, an $S$-affine function is a function which coincide on every portion $Y(x)$, $x\in S_X$, with a truly affine function whose slope and intercept might depend on $x$. 
Obviously, affine functions are \(S\)-affine for any \(S\). Moreover \(\aff(S)\) is a vector space.

\begin{remark}\label{CRR}
{\rm Note that the functions \(\varphi, h\) in (\ref{Qaff}) are uniquely defined from \(\psi\) because $|Y_S(x)| \ge 2$, which is due to the fact that $S$ is assumed 1-erased. Indeed, take $y_1,y_2 \in Y_S(x)$ with $y_1 \neq y_2$. 
Being $\psi$ $S$-affine, we have in particular that
\[ \psi(y_i) = \varphi (x) + h(x) (y_i -x),\quad i=1,2. \]
This is a linear system of two equations with two unknowns $h(x),\varphi(x)$, that can be solved explicitly giving
\begin{equation} h(x) = \frac{\psi(y_1) - \psi(y_2)}{y_1 - y_2}, \quad \varphi(x) = \frac{x-y_2}{y_1-y_2}\psi(y_1) + \frac{y_1 - x}{y_1 - y_2}\psi(y_2).  \label{eq-CRR}\end{equation}}
\end{remark}

The other important ingredient of this section is the new notion of 2-net, which we introduce in the following definition.

\begin{definition}[2-net]\label{def:2-net}
A 1-erased set $A \subset X \times Y$  is called \emph{2-net} if every $A$-affine function is affine.
\end{definition}

Intuitively, a 2-net is a subset of $X \times Y$ where the WEP is defined without ambiguity, i.e. modulo
an affine function so that a 2-net has intrinsically the corresponding two degrees of freedom, whence its name.The following property follows from the definition of 2-net:

\begin{proposition}
Every 2-net $A$ is connected.
\end{proposition}
\begin{proof}
Assume by contradiction that $A$ is not connected, i.e. there exist at least two disjoint connected subclasses, say $A_1, A_2, \ldots $. Consider some $A_i$-affine function $\psi_i$, for $i \ge 1$. Since $A_i$ is a 2-net for all $i$, 
we have $\psi_i (y) = \alpha_i + \beta_i y$, for all $y \in (A_i)_Y$, for some constant $\alpha_i, \beta_i \in \mathbb R$. Define $\psi(y) := \sum_{i \geq 1} \psi_i (y) \mathbf 1_{\{y \in (A_i)_Y\}}$. 
This is an $A$-affine function, which is not affine. Therefore, $A$ must be connected. 
\end{proof}

The next lemmas give, respectively, an equivalent characterization of 2-nets and a sort of stability property, according to which adding points to a given 2-net preserves the 2-net structure.

\begin{lemma}
A set $A$ is a 2-net if and only if for all $\psi \in \aff(A)$ as in (\ref{Qaff}) we have 
\[ h(x)=h(x'), \quad \forall (x,x') \in (A_X)^2.\]
\end{lemma}

\begin{proof}
The direct implication is obvious. Let us prove the other direction.
Let $\psi \in \aff(A)$ and let $\beta$ denote the common value of $h(x)$ for $x \in A_X$. Then $\psi(y)=\varphi(x)+\beta (y-x)$ for some function $\varphi(x)$, 
or yet $\varphi(x)- \beta x = \psi(y)-\beta y$. For $x$ and $x'$  with $Y_A(x) \cap Y_A(x') \neq \emptyset$ this yields $\varphi(x)-\beta x=\varphi(x')-\beta x'$. 
Since every two points in $A$ are connected, $\varphi(x)-\beta x=\alpha$ on $A$ for some constant $\alpha$ and the proof is completed.
\end{proof}

\begin{lemma} \label{extends2net}
Let $A,B$ be two 2-nets such that $ | A_Y \cap B_Y | \geq 2$. Then $ A \cup B$ is a 2-net.
\end{lemma}

\begin{proof}
Let $\psi$ be a $A \cup B$-affine function. Let $\alpha + \beta y$ be the affine function matching $\psi$ on $A_Y$ and let 
$\gamma + \delta y$ be the affine function matching $\psi$ on $B_Y$. Since $ | A_Y \cap B_Y | \geq 2$ we have $\alpha=\gamma$ and $\beta=\delta$, hence $\psi$ is affine on $A_Y \cup B_Y = (A\cup B)_Y$.
\end{proof}

The following two examples clarify the relation between 2-nets and the 2-link property.

\begin{example}
{\rm Any subset $A$ satisfying (\ref{2LP}) with equality is a 2-net. Indeed, assume that there exists a numbering $A_X =(x_n)_{n\ge 1}$ such that
\begin{equation}\label{2LPeq} | Y(x_n) \cap \bigcup_{i=1}^{n-1} Y(x_i) | = 2 ,\quad n \ge 1.\end{equation}
To show that $A$ is a 2-net we proceed by induction. First, $\{x_1\} \times Y(x_1)$ is trivially a 2-net. Assume now that $A_{n-1} := \{(x_i,y): y \in Y(x_i), i=1,\ldots,n\}$ is a 2-net. 
Since (\ref{2LPeq}) holds for all $n$, we can apply Lemma \ref{extends2net} yielding that $A_{n-1} \cup (\{x_n\} \times Y(x_n))$ is a 2-net. Therefore, $A$ is a 2-net. In particular, the 2-net described in Example \ref{infinite-2LP} provides an example of an infinite 2-net.}
\end{example}
\begin{example}\label{patras2}
{\rm The support described in Example \ref{patras} is also a 2-net. Indeed, both sets $\{x_1,x_2\}$ and $\{x_3,x_4\}$ are 2-nets and they are connected to each other with exactly two links. Hence, their union is a 2-net by Lemma \ref{extends2net}.}
\end{example}

\subsection{$S$-maximal 2-nets}
%%%%%%%%%%%%%%%
In this section we introduce the notion of $S$-maximal 2-net and give some properties that will reveal useful later in this section.

\begin{definition}[Maximal 2-net]\label{max2net}
A 2-net $A \subset S$ is $S$-maximal if for any 2-net $A' \subset S$ such that $A \subset A'$, we have $A=A'$.
\end{definition}

\begin{proposition} $S$-maximal 2-nets exist.\end{proposition}

\begin{proof} The existence of maximal 2-nets is guaranteed by an application of Zorn's Lemma (e.g. 1.7 in \cite{AB}). Let $\mathcal A$ denote the class of all 2-nets in $S$ and let $\mathcal A'$ any subclass of $\mathcal A$, totally ordered with respect to set inclusion, i.e. for any two elements $A'_1, A'_2 \in \mathcal A'$ we have either $A'_1 \subset A'_2$ or $A'_2 \subset A'_1$. We need to prove that $\mathcal A'$ has an upper bound in $\mathcal A$. Consider $A_0 := \cup_{A' \in \mathcal A'} A'$, which by definition contains any 2-net in $\mathcal A'$. To conclude, it remains to show that $A_0$ is a 2-net. In order to do so, take an $S$-affine function $f$. By definition, $f$ coincide with an affine function on every 2-net $A'$ with possibly different intercepts and slopes $\varphi_{A'},h_{A'}$. Consider two 2-nets in $\mathcal A'$, say $A'_1,A'_2$. Since they are totally ordered, we have $A'_1 \subset A'_2$ or the opposite. Both situations imply $\varphi_{A_1'} = \varphi_{A_2 '}$ and $h_{A_1'} = h_{A_2 '}$. Therefore, being $A' _i$, $i=1,2$, arbitrary, we have that slopes and intercepts of $f$ are the same on every 2-net $A' \in \mathcal A'$. Since this is true for all $S$-affine functions $f$, we conclude that $A_0$ is a 2-net and Zorn's Lemma applies. 
\end{proof}

\begin{proposition}\label{properties-max2net}
Let $A, B \subset S$ be two $S$-maximal distinct 2-nets. The following properties holds:
\begin{enumerate}
\item[(i)] for all $z \in S_X \setminus A_X$, we have
$$ \quad |Y(z) \cap Y(A) | \leq 1 ;$$
\item[(ii)] $A_X \cap B_X =\emptyset$ and $| A_Y \cap B_Y | \leq 1$.
\end{enumerate}
\end{proposition}

\begin{proof}
Property (i) is a direct consequence of Lemma
\ref{extends2net} and Definition \ref{max2net}. 

Regarding the properties in (ii): assume that there exists $z \in A_X \cap B_X$. Hence $A \cup B$ is connected since both are and they have a point in common. 
Take an $S$-affine function $\psi$. Since both $A$ and $B$ are 2-nets, such a function is affine on $A$ and $B$ separately, with slopes and intercepts, respectively, $\varphi_A, h_A$ and $\varphi_B, h_B$. 
Moreover, $\varphi_A = \varphi_B$ and $h_A =h_B$ since $A_X$ and $B_X$ have the point $\{z\}$ in common: indeed since every 2-net is 1-erased, we have
$|Y(z) \cap A_Y| \geq 2$ and $|Y(z) \cap B_Y| \geq 2$. Therefore, $\psi$ is affine on $A \cup B$ and being $\psi$ arbitrary we have that $A \cup B$ is a 2-net, so contradicting the assumption that $A$ and $B$ are $S$-maximal. Hence, $A_X \cap B_X = \emptyset$. Now, assume that $| A_Y \cap B_Y | \ge 2$. Proceeding as in the proof of Lemma \ref{extends2net}, we can prove that $A \cup B$ is a 2-net strictly bigger than both $A$ and $B$ since they are disjoint, so contradicting the fact that they are $S$-maximal.
\end{proof}

\begin{lemma}\label{add-one-branch}
Let $T \subset X \times Y$ be a $1$-erased and connected set with decomposition $(A_i)_{i=1}^k$ in $k$ maximal 2-nets with $k \ge 1$.
Let $x \in (A_1)_X$ and consider the set $T'=T \cup \{(x, y)\}$ where $(x,y) \notin T$, and $y \in T_Y$.
Then $T'$ decomposes in at most $k$ maximal 2-nets.
\end{lemma}

\begin{proof}
First, we observe that the set $A_1 \cup \{(x,y)\}$ is a 2-net in $T'$. By Definition \ref{def:2-net} any 2-net in $T$ is also a 2-net in $T'$, hence all the sets $A_i$, for $i=2,\ldots ,k$, are 2-nets in $T'$.
Since any 2-net in $T'$ is included in a maximal 2-net in $T'$, there are at most $k$ such sets.
\end{proof}

\begin{remark}
{\rm In the situation of the lemma above, the new cardinality may be any number between $1$ and $k$, depending on the connections between the 2-nets $A_i$: \begin{itemize}
\item If $y \in (A_1)_Y$, then the $A_1'=A_1 \cup \{(x,y)\}$ affine functions are exactly the $A_1$ affine ones, and $T'$ decomposes in the $k$ maximal 2-nets $A_1', A_2,\ldots,A_k$.
\item If $y \in (A_2)_Y$ and if there is a point $z \neq y$ in the intersection $B=(A_1)_Y \cap (A_2)_Y$, then $A_1' \cup A_2$ is a 2-net. Either the $Y$-sections of the other 2-nets $A_i, \; i \geq 3$ have a single intersection  with the $Y$-section of this new 2-net, and the cardinality of the decomposition of $T'$ is $k-1$, or the $T'$ maximal 2-net which contains $B$ contains other sets $A_i$, the cardinality of the decomposition of $T'$ is strictly less than $k-1$, possibly reaching $1$. \end{itemize}}
\end{remark}

%%%%%%%%%%%%%%%%%%%%%%%%%%%%%%%%%%%%%%
\subsection{Saturated 2-nets, deadlocks and the WEP}\label{the-wep-on-2-nets}
%%%%%%%%%%%%%%%%%%%%%%%%%%%%%%%%%%%%%%

In this section we study the relation between the WEP and the new notion of deadlocks of $S$ introduced just below.
By definition, the WEP for a given function \(f\) is defined only up to an \(S\)-affine function. 
Recall that we say that $\textrm{WEP}(f)$ holds on a set $A$ if we have
\begin{equation} f(x,y) = \varphi_A (x) + h_A (x) (y-x) -\psi_A (y), \quad (x,y) \in A, \label{WEPfA}\end{equation}
for some functions $\varphi_A , h_A, \psi_A$.
Since on 2-nets all \(S\)-affine functions are affine, we have immediately the following proposition, whose proof is straightforward and therefore it is omitted.

\begin{proposition}\label{uptoaffine}
Let $f:X \times Y \to \mathbb R$ be a given function. Assume that $\textrm{WEP}(f)$ holds on a 2-net $A$. Then the corresponding decomposition (\ref{WEPfA}) is defined up to an affine function on $A$. \end{proposition}

\noindent Let us introduce now the following definition of \emph{deadlock}, which will be used in the main result of this section. The importance of such a notion will be illustrated in Example \ref{italiens}. We recall that in our setting a mesh $M$ is any subset of $S$ with $|M_X|=1$. We also use the notation $M(x) = M$ if $M_X = \{x\}$, i.e. $M(x) = \{x\} \times Y(x)$.

\begin{definition}\label{deadlock}
Let $S \subset X \times Y$. We say that any triplet $(T,x_0,y_0)$, where $T \subset S$ and $(x_0,y_0) \in S$ is a \emph{deadlock} in $S$ if $|(M(x_0) \cap T)_Y| > 1$ and the following two properties hold: \begin{enumerate}
\item[(i)] $x_0 \in T_X$, $y_0 \in T_Y$, while $(x_0, y_0) \notin T$, 
\item[(ii)] every $T$-affine function which is null on $(M(x_0) \cap T)_Y$, is also null at $y_0$.
\end{enumerate}
\end{definition}

\begin{remark}
{\rm Observe that any 2-net $T$ satisfying the property (i) above satisfies (ii) for free. Indeed, $T$ being a 2-net any $T$-affine function, say $\psi$, is actually affine. Moreover $\psi$ is null on $(M(x_0)\cap T)_Y$, which contains at least two distincts points as, by definition of $2$-net, $T$ is also 1-erased. Hence $\psi$ is null everywhere in $T_Y$ and, in particular, at the point $y_0$.}  
\end{remark}

In view of the remark above, we will say that a 2-net $T$ is \emph{saturated} if the property (i) in Definition \ref{deadlock} never holds in $T$, i.e. for all pairs $(x_0,y_0)\in S$ with $x_0 \in T_X$ and $y_0 \in T_Y$ one has $(x_0,y_0)\in T$. 

\begin{example}[Example of deadlock] \label{Pallot}
{\rm Consider the following set $S$, which is taken from R. Pallotini's dissertation \cite[Section 4.5]{pallottini}.
Let $T = S \setminus \{(x_4, y_6)\}$. 
\begin{figure}[H]	
\centering
\includegraphics[scale=1]{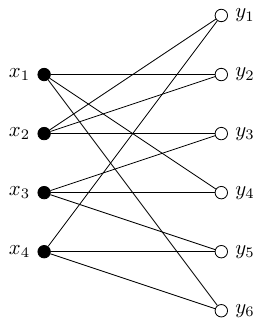}
\end{figure}
Then there is a critical value of $y_6$, denoted $y_6 ^*$, for which $T$ is a deadlock. Indeed consider $T$-affine functions which are null on $(M(x_4)\cap T)_Y=\{y_1,y_5\}$. Such functions are affine on $\{y_1,y_2,y_3\}$, $\{y_3,y_4,y_5\}$ and $\{y_2,y_4,y_6\}$, so they can be parameterized by their value, say $u$, at $y_3$. We plot in the following figure two examples of $T$-affine functions for $u=2$ (the red solid line) and $u=6$ (the blue solid line). 

\begin{figure}[H]	
\centering
\includegraphics[scale=1]{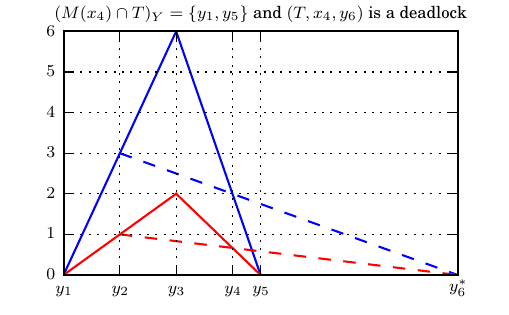}
\end{figure}

We can see that the dashed line crosses the $x$-axis at the same level whatever the value of $u$, which is the critical value $y^*_6$. In this case any $T$-affine function which is null on $(M(x_4)\cap T)_Y=\{y_1,y_5\}$ is also null at the point $y_6 (=y^*_6)$.

To prove that the crossing point does not depend on $u$, let $z$ be its value. By Thales' theorem we have that $\frac{z_2}{z_4}=\frac{y_2-z}{y_4-z}$ where $z_2$ and $z_4$ are the values of the $T$-affine function at the points $y_2$ and $y_3$, so that
$z_2 = u \frac{y_2-y_1}{y_3-y_1}$ and $z_4 = u \frac{y_5-y_4}{y_5-y_3}$, which yields that $\frac{z_2}{z_4}=\frac{y_2-y_1}{y_3-y_1}\cdot \frac{y_5-y_3}{y_5-y_4}$ does not depend on $u$, and eventually we get the value of $y^*_6$ by solving the equation
\[\frac{y_2-y_1}{y_3-y_1}\cdot \frac{y_5-y_3}{y_5-y_4}=\frac{y_2-y^*_6}{y_4-y^*_6}.\]
Note that, by the same reasoning, when $y_6 \neq y^*_6$, any $S$-affine function which is null on $M(x_4)_Y=\{y_1,y_5, y_6\}$ is necessarily null on $(y_2, y_4)$, hence everywhere. This proves that $S$ is a 2-net in the non critical case. And so only in this case since there are non-null $S$-affine functions in the critical case.}
\end{example}

We can finally state the main result of this section.

\begin{theorem} \label{thm2net}
Let $S$ be any subset of $X \times Y$. If the WEP holds for $S$, then $S$ does not contain any deadlock. Conversely, if $S$ does not contain any deadlock, there is an increasing sequence of sub-sets $(S_n)_{n \ge 1} \subset S$ such that the following properties hold:\begin{enumerate} 
\item[(i)] $|(S_n)_X|=n$ for all $n \ge 1$, and each $S_n$ is decomposed in finitely many maximal 2-nets;
\item[(ii)] the WEP holds on each $S_n$, for $n \ge 1$;
\item[(iii)] $\cup_{n \ge 1} S_n = S$.
\end{enumerate}
\end{theorem}

\begin{proof} 
We prove first that the WEP for $S$ implies that it does not contain any deadlock. Let us proceed by contradiction and consider a deadlock $(T,x_0,y_0)$ as in Definition \ref{deadlock} and let $f=\mathbf 1_{\{(x_0,y_0)\}}(x,y)$. As WEP($f$) holds, let $(\varphi, h, \psi)$ be any of its decomposition. In particular since $f \equiv 0$ on $T$, $\psi$ is a $T$-affine function such that
$$\psi(y) = \varphi(x)+h(x)(y-x).$$
Consider now the affine function $y \mapsto \varphi(x_0)+h(x_0)(y-x_0)$. It can be written as $\varphi(x_0)+h(x_0)(x-x_0)+h(x_0)(y-x)$ so that the triplet $(\varphi', h', \psi')$ given by
\[ \varphi' = \varphi - (\varphi(x_0)+h(x_0)(x-x_0)), \quad h'= h - h(x_0), \quad \psi'=\psi - (\varphi(x_0)+h(x_0)(y-x_0)),\] is also a decomposition of $f$. Moreover we have $\varphi'(x_0)=h'(x_0)=0$, so that the $T$-affine function $\psi'$ is null on the set $M(x_0)_Y$ in $T_Y$, which entails by the deadlock property (ii) in Definition \ref{deadlock} that $\psi'(y_0)=0$. Therefore $\varphi'(x_0)+h'(x_0)(y_0-x_0)-\psi'(y_0)=0$ whereas $1=f(x_0, y_0) = \varphi'(x_0)+h'(x_0)(y_0-x_0)-\psi'(y_0)$, whence a contradiction. This completes the proof of the first part of this theorem.\medskip

To prove the second part, we need to show that, under the no-deadlock assumption, there exists a sequence of sets $S_n \uparrow S$ fulfilling the properties (i)-(ii)-(iii) in the statement. By Proposition \ref{decompWEP}, we can assume without loss of generality that $S$ is connected (cf. Definition \ref{connected}). Let $f:X \times Y \to \mathbb R$ be an arbitrary function. We prove that $f$ satisfies locally the WEP over a suitable sequence of subsets $S_n \subset S$ with the announced properties, whose recursive construction goes as follows.

Let $S_1 = M(x_1)$. The WEP holds for $f$ on $S_1$ by setting $(\varphi(x_1), h(x_1))=(0,0)$ and $\psi(y)= f(x_1,y)$ on $M(x_1)_Y$. Moreover $|(S_1)_X|=1$ and $S_1$ is a maximal 2-net.

Now let us assume that $f$ satisfies the WEP on $S_n$. Hence either $S_n=S$, and we are done, or there is some other point $x_{n+1}$ such that $C:=M(x_{n+1})_Y \cap (S_n)_Y \neq \emptyset$. Indeed, if $C$ was empty, then $S$ would not be connected, which contradicts our initial assumption. Let $(A_i)_{1 \leq i \leq k}$ be the decomposition of $S_n$ in maximal 2-nets.
By Corollary \ref{extended}, $C$ has at most two points in the $Y$-section of each $A_i$.

Now we are going to extend WEP($f$) to $S_{n+1} = S_n \cup M(x_{n+1})$. The first step consists in extending it to $S_n \cup \{(x_{n+1},y):y \in C\}
$. It is useful to distinguish between two cases:\begin{enumerate}
\item[(a)] Assume $|C|=1$. Then for the only point $y \in C$ it suffices to set $\varphi(x_{n+1})=f(x_{n+1}, y) + \psi(y)$ and $h(x_{n+1})=0$.

\item[(b)] Consider now the situation $|C| \geq 2$. There are two possible sub-cases:
\begin{enumerate}
\item[(b.1)]First, assume there are two distinct points $y_1, y_2$ in $C$ which belong to the $Y$-section of the same maximal 2-net, which we can assume to be $A_1$ possibly after relabelling. We can extend WEP($f$) to the set $\{(x_{n+1},y_j): j=1,2\}$ by means of the formulas in \eqref{eq-CRR}, yielding the values $\varphi(x_{n+1})$ and $h(x_{n+1})$. Now, either there are no more right-hand points in $M(x_{n+1})$ and we are done, or there is another point $y_3$. Since $A_1$ is saturated, $y_3$ cannot belong to $(A_1)_Y$. Possibly after relabelling, we can assume that $y_3 \in (A_2)_Y$. 

Since $S$ has no deadlock, we can pick an $S_n$-affine function $\chi$ such that $\chi(z)=0$, for all $z \in (A_1)_Y$, and $\chi(y_3)+\psi_n(y_3)=\varphi(x_{n+1}) + h (x_{n+1}) (y_3-x_{n+1}) - f(x_{n+1}, y_3)$. We modify then WEP($f$) on $S_n \setminus A_1$ by adding to $\varphi$ and $h$ the decomposition of $\chi$ as a $S_n$-affine function. Notice that the WEP is preserved due to Proposition \ref{uptoaffine}. 

Then we decompose in maximal 2-nets the new set $T:= S_n \cup \{(x_{n+1},y_j): j=1,2,3\}$ to which WEP($f$) has been extended. Note that, due to Lemma \ref{add-one-branch}, such a decomposition has a cardinality less than or equal to $k$.

\item[(b.2)] Second, assume that there is at most one point in the intersection of $C$ and the $Y$-section of any maximal 2-net in the decomposition $(A_i)_{i=1}^k$ of $S_n$. In this case, we pick two distinct points $y_1, y_2 \in C$, and obtain the decomposition of the WEP $(\varphi(x_{n+1}), h(x_{n+1}))$ by the formulas \eqref{eq-CRR}. The key point is now to observe that the binomial mesh $M_{bin}=\{(x_{n+1}, y_1), (x_{n+1}, y_2) \}$ forms a 2-net which will be maximal in the decomposition of the set $T :=S_n \cup M_{bin}$. This is a consequence of the fact that $\aff (T) \subset \aff (S_n)$, as the affine functions on $M(x_{n+1})$ are affine on $A_1$ as well.  

Now, either there are no more right-hand points in $M(x_{n+1})$ and we are done, or there is another one, say $y_3$, and since $S$ has no deadlock, we
can choose a $T$-affine function $\chi$ such that $\chi((M_{bin})_Y)=0$ and $\chi(y_3)+\psi_n(y_3)=\varphi(x_{n+1}) + h(x_{n+1}) (y_3-x_{n+1}) - 
f(x_{n+1}, y_3)$. Then, we modify WEP($f$) on $T \setminus M_{bin}$ by adding to $\varphi$ and $h$ the decomposition of $\chi$ as a $T$-affine 
function. 

Finally, we observe that in this sub-case the decomposition in maximal 2-nets of the new set $T$ has a cardinality less than or equal to $k+1$. This is due to the same argument as above, except that now we have $k+1$ (instead of $k$ as before) because of the additional 2-net $M_{bin}$.
\end{enumerate}
Now, either there are no more right-end points in $M(x_{n+1})$ (which is necessarily the case if there is a single maximal 2-net in the decomposition, by Corollary \ref{extended}), in which case we are done. Otherwise there is another point $y \in M(x_{n+1})_Y$, and we re-iterate the steps above, extending in this way WEP($f$) to $S_n \cup \{(x_{n+1}, y), y \in C \}$.
\end{enumerate}
The second and final step consists in extending WEP($f$) to the rest of $M(x_{n+1})$ by setting $\psi_{n+1}(y) = \varphi(x_{n+1}) + h(x_{n+1}) (y-x_{n+1})-f(x_{n+1}, y)$ for the right-hand points $y$ of this set.

Finally, we observe that the new set $S_{n+1}$ satisfies $|(S_{n+1})_X|=n+1$ and it decomposes in finitely many maximal 2-nets. Therefore, the proof is complete.
\end{proof}

\begin{remark}
\emph{Notice that if $S \subset X \times Y$ is not necessarily 1-erased, we can always apply the main theorem above to the set $\mathcal E_{1,x}(S)$, which is 1-erased.}
\end{remark}

\begin{remark}
\emph{The main theorem above can be reformulated as follows when $S_X$ is finite: the WEP holds on $S$ if and only if $S$ does not contain any \emph{deadlocks}. 
Hence, it seems that in our formulation the deadlocks play the same role that the cycles have in characterizing extremality of measures with given marginals (without the martingale property as in \cite{denny}), i.e. both are forbidden patterns in the supports of their respective extremal measures. A critical difference is that the numerical values of $y$'s and not only the way the points are connected seem to matter in the martingale case (compare Example \ref{italiens}).}
\end{remark}

Here is the statement and proof of the ``extended intersection lemma'' that has been used in the second part of the proof above. 

\begin{corollary}[Extended intersection Lemma] \label{extended}
Assume that the WEP holds for $S$ and let $A$ be a 2-net in $S$. Then for any $z \in S_X \setminus A_X$, $|Y(z) \cap A_Y| \leq 2$.
\end{corollary}

\begin{proof}
Assume the contrary, then the set $A \cup \{(z,y_i): i=1,2\}$ where the $y_i$'s belong to the intersection $Y(z) \cap A_Y$, is a 2-net, and therefore is saturated. Hence there can not
be a third point in the intersection and we have a contradiction.
\end{proof}

We conclude this section with some simple examples illustrating the content of the main Theorem \ref{thm2net} and the role played by the ``no deadlock'' assumption.

\begin{example}
{\rm Consider a very simple situation with $X=\{x_1,x_2\}$, $Y=\{y_1,y_2,y_3\}$ and where the paths in $S$ are given in the following graph:
\begin{figure}[H]	
\centering
\includegraphics[scale=1]{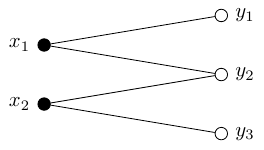}
\end{figure}

\noindent This clearly satisfies (\ref{2LP}), so that WEP holds. Let us verify that it does not contain any deadlock. Consider any triple $(T,x_0,y_0)$ with $|(M(x_0)\cap T)_Y| >1$, $(x_0, y_0) \in (T_X \times T_Y) \setminus T$. We need to show that we can find a $T$-affine function, which is null on $(M(x_0) \cap T)_X$ and not at the point $y_0$. In this example, the only triples with the properties above are $(T,x_i,y_j)$ for $T \in \{S,M(x_i)\}$ and $j=3$ (resp. $1$) if $i=1$ (resp. $2$). For each of them, we can check that property (ii) in Definition \ref{deadlock} is not satisfied. For instance, consider $(S,x_1,y_3)$ and take any $S$-affine function $\psi (y) = \alpha(x) + \beta(x)y$ for $y \in S_Y$ and $x \in X(y)$, which is null on $M(x_1) = \{y_1,y_2\}$. This implies $\alpha(x_1)+ \beta(x_1)y_1= \alpha(x_1)+ \beta(x_1)y_1 = 0$, so that $\alpha(x_1)=\beta(x_1)=0$. Now, we also have $\psi(y_2)= \alpha(x_2) + \beta(x_2)y_2 =0$, hence $\alpha(x_2) = -\beta(x_2)y_2$. Therefore $\psi(y_3) = \alpha(x_2) + \beta(x_2) y_3 = \beta(x_2) (y_3 -y_2)$, so that we can clearly have $\beta(x_2)\neq 0$. This means that $(S,x_1,y_3)$ is not a deadlock. We can similarly get to the same conclusion for the other triples.}
\end{example}

\begin{example}\label{patrasWEP}
{\rm The set in Example \ref{patras} is a maximal 2-net that fulfills the condition in Theorem \ref{thm2net}. Thus WEP holds. To see this, let us consider an arbitrary function $f$ and look for a triple $(\varphi, h, \psi)$ such that \eqref{WEPf} holds true. Moreover we will use the notation $\varphi_i = \varphi(x_i)$, $h_i =h(x_i)$ and $\psi_j = \psi(y_j)$ for all $i=1,\ldots,4$ and $j=1,\ldots,6$.

First, note that we can always assume that $(\varphi_1, h_1)=(0,0)$, so that the relation \eqref{WEPf} on the mesh $M(x_1)$ gives the values $(\psi_1,\psi_2,\psi_3)$. The same relation on the branches of $M(x_2)$ ending in $M(x_1)_Y$ gives in turn $(\varphi_2, h_2)$, and the last branch of $M(x_2)$ gives $\psi_4$. Turning to $M(x_3)$, the equation \eqref{WEPf} at $(x_3, y_4)$ gives $\varphi_3$ as a function of $h_3$, which we substitute in the expression for \eqref{WEPf} the two remaining paths in $M(x_3)$, hence obtaning $h_3(y_5-y_4) = \psi_5+ \text{known terms}$, and  $h_3(y_6-y_4) = \psi_6+ \text{known terms}$. The same analysis for  $M(x_4)$ yields $h_4(y_5-y_3) = \psi_5+ \text{known terms}$, and  $h_4(y_6-y_3) = \psi_6+ \text{known terms}$. Finally substituting $\psi_5$ and $\psi_6$ we get a linear system in $(h_3,h_4)$, whose determinant is given by $(y_5-y_4)(y_3-y_6)-(y_3-y_5)(y_6-y_4)$. Such a determinant is not null, because considering $y_6$ as a variable there is at most one value which makes it zero, which is $y_6=y_5$, and the points $(y_i)_{1 \leq i \leq 6}$ are assumed to be distinct. As consequence, since $f$ is arbitrary the WEP($f$) is satisfied.

We recall that none of the sufficient conditions previously discussed in Section \ref{suff} work in that example.}
\end{example}

\begin{example}\label{italiens}
{\rm 
Let us revisit Example \ref{Pallot}. In the critical case, i.e. $y_6 =y_6 ^*$, there is a deadlock, hence the WEP should not hold according to Theorem \ref{thm2net}. Let us investigate directly the WEP. 
To simplify the notation, we denote $f_{i,j} := f (x_i,y_j)$ for all $i,j$. We can assume without loss of generality that $\varphi_1= h_1 = 0 $ on the mesh $M(x_1)$, giving
\begin{align*}
 f_{1,6}+ \psi_6 = f_{1,4}+ \psi_4 = f_{1,1}+ \psi_2 =0.
\end{align*}
On the second mesh $M(x_2)$, we have
\begin{align*}
\varphi_2 + h_2 (y_3-x_2) &= f_{1,3} + \psi_3, \\
\varphi_2 + h_2 (y_2-x_2) &= f_{2,2} + \psi_2 = f_{2,2} - f_{1,1}, \\
\varphi_2 + h_2 (y_1-x_2) &= f_{2,1} + \psi_1.
\end{align*}
So by choosing $\psi_1$ as a parameter, the last two equations above give
\begin{align*}
\varphi_2 &= \frac{(y_1-x_2) (f_{2,2}-f_{1,1}) - (y_2-x_2) (f_{2,1}+\psi_1)}{y_1-y_2}, \\
h_2 &= \frac{(f_{2,2}-f_{1,1}) - (f_{2,1}+\psi_1)}{y_2-y_1},
\end{align*}
whence exploiting the remaining first one we get
\begin{equation} \label{psi3} \psi_3 = \frac{y_2-y_3}{y_2-y_1} \psi_1 + \frac{y_3-y_1}{y_2-y_1} (f_{2,2}-f_{1,1}) +\frac{y_2-y_3}{y_2-y_1} (f_{2,1}-f_{2,3}).\end{equation}
At this stage we have obtained $(\varphi_2, h_2, \psi_3)$ as functions of $\psi_1$. In exactly the same way for the mesh $M(x_3)$, while taking $\psi_3$ has a parameter, we have
\begin{align*}
\varphi_3 &= \frac{(y_3-x_3) (f_{3,4}-f_{1,4}) - (y_4-x_3) (f_{3,3}+\psi_3)}{y_3-y_4}, \\
h_3 &=  \frac{(f_{3,4}-f_{1,4}) - (f_{3,3}+\psi_3)}{y_4-y_3},
\end{align*}
so that
\begin{equation}\label{psi5} \psi_5 = \frac{y_4-y_5}{y_4-y_3} \psi_3 + \frac{y_5-y_3}{y_4-y_3} (f_{3,4}-f_{1,4}) +\frac{y_4-y_5}{y_4-y_3} (f_{3,3}-f_{3,5}).\end{equation}
We have now  $(\varphi_3, h_3, \psi_5)$ as functions of $\psi_3$, hence as functions of $\psi_1$ as well.
Working out the mesh $M(x_4)$, and using this time $\psi_1$ has a parameter, we obtain
\begin{align*}
\varphi_4 &= \frac{(y_1-x_4) (f_{4,6}-f_{1,6}) - (y_6-x_4) (f_{4,1}+\psi_1)}{y_1-y_6}, \\
h_4 &=  \frac{(f_{4,6}-f_{1,6}) - (f_{4,1}+\psi_1)}{y_6-y_1},
\end{align*}
yielding
\begin{equation}\label{psi5x4} \psi_5 = \frac{y_6-y_5}{y_6-y_1} \psi_1 + \frac{y_5-y_1}{y_6-y_1} (f_{4,6}-f_{1,6}) +\frac{y_6-y_5}{y_6-y_1} (f_{4,1}-f_{4,5}),\end{equation}
hence getting $(\varphi_4, h_4, \psi_5)$ in terms of $\psi_1$. 

Now, we have two different expressions for $\psi_5$ (since they involve $f$ evaluated along different paths), so the only way to reconciliate them is to adjust the value of $\psi_1$.
Substituting the expression \eqref{psi3} of $\psi_3$ in equation \eqref{psi5}, we obtain 
\begin{eqnarray*}
 \psi_5 &=& \frac{y_4-y_5}{y_4-y_3} \cdot \frac{y_2-y_3}{y_2-y_1} \psi_1 + \frac{y_3-y_1}{y_2-y_1} (f_{2,2}-f_{1,1}) +\frac{y_2-y_3}{y_2-y_1}( f_{2,1}-f_{2,3}) \\
 && + \frac{y_5-y_3}{y_4-y_3} (f_{3,4}-f_{1,4}) +\frac{y_4-y_5}{y_4-y_3} (f_{3,3}-f_{3,5}).
 \end{eqnarray*}
Therefore we need to have equality between the expression for $\psi_5$ just above and the one in \eqref{psi5x4}, which yields
\begin{align*} 
\frac{y_6-y_5}{y_6-y_1} \psi_1 + & \frac{y_5-y_1}{y_6-y_1} (f_{4,6}-f_{1,6}) +\frac{y_6-y_5}{y_6-y_1}( f_{4,1}-f_{4,5} ) \\ 
 = & \frac{y_4-y_5}{y_4-y_3} \cdot \frac{y_2-y_3}{y_2-y_1} \psi_1 + \frac{y_3-y_1}{y_2-y_1} (f_{2,2}-f_{1,1}) +\frac{y_2-y_3}{y_2-y_1} (f_{2,1}-f_{2,3}) \\
& + \frac{y_5-y_3}{y_4-y_3} (f_{3,4}-f_{1,4}) +\frac{y_4-y_5}{y_4-y_3} (f_{3,3}-f_{3,5}).
\end{align*}
The equation above has a solution if and only if
$$ \frac{y_6-y_5}{y_6-y_1} \neq \frac{y_4-y_5}{y_4-y_3}\cdot  \frac{y_2-y_3}{y_2-y_1} ,$$
otherwise WEP($f$) cannot be satisfied.

Working out the critical condition $ \frac{y_6-y_5}{y_6-y_1} = \frac{y_4-y_5}{y_4-y_3} \cdot \frac{y_2-y_3}{y_2-y_1} $ by viewing it as an equation in $y_6$, we get after some manipulations the critical case $y_6 = y_6 ^*$ of Example \ref{Pallot}.

Our computations also show that in the non-critical case there is no deadlock, which is not completely obvious as using the definition of deadlock would require a careful inspection of each subset $T$ of $S$.}
\end{example}

\begin{remark}
{\rm As the above example suggests, Theorem \ref{thm2net} does not really simplify the investigation of the WEP in practice, due to the fact that the no deadlock property should be verified for every subset $T$ of $S$. It rather works in the other direction: if some deadlock is found by direct considerations, the WEP cannot hold. Its theoretical value is to translate the WEP in a property of the locally affine functions of the subsets of $S$, which illustrates the importance of such functions in this context.}
\end{remark}

\subsection{Some complementary results on 2-nets and WEP}
%%%%%%%%%%%%%%%%%%%%%%%%%%%%%%%

In this section we gather some consequences and complements of Theorem \ref{thm2net} and the extended intersection lemma (Corollary \ref{extended}).
Indeed, using the latter we are now able to prove that when $|S_Y|$ is small enough, full erasability and the WEP are equivalent:

\begin{proposition}\label{5-erasability}
Assume \(|S_Y|\leq5\). Then the WEP holds for $S$ if and only if $S$ is fully erasable.
\end{proposition}

\begin{proof}

We already know (cf. Proposition \ref{prop:fullyWEP}) that, for a given set $S$, full erasability implies the WEP. Hence, it suffices to prove the opposite implication. We can assume, without loss of generality, that $S$ is 2-erased, so that $|Y(x)|\geq 3$, for all $x \in S_X$.  Let $x_1 \in S_X$. We distinguish three different cases.\medskip\\
(i) Case $|S_Y|=3$: Then $|Y(x_1)|=3$ and by the intersection Lemma \ref{2points} there can not be another point in $S_X$, so that the paths of $M(x_1)$ are isolated
and $S$ is fully erasable.\medskip\\
(ii) Case $|S_Y|=4$: If $|Y(x_1)|=4$ then we can conclude as above. If $|Y(x_1)|=3$, let $x_2 \in S_X$. By the intersection lemma, $|Y(x_2)|=3$, $Y(x_2)$ intersects $Y(x_1)$ in exactly two points.
Therefore $M(x_1)\cup M(x_2)$  is a 2-net and by the extended intersection Lemma \ref{extended}  there can not be another point in $S_X$. Now since the mesh $M(x_2)$ has an isolated branch, i.e. a path $(x_2,y)$ with $X(y_2)=\{x_2\}$, it is erasable,
and we can then erase $M(x_1)$ whose paths are isolated.\medskip\\
(iii) Case  $|S_Y|=5$: If $|Y(x_1)|=5$, we can conclude as in case (i) and if $|Y(x_1)|=4$ the proof is the same as in the previous case (ii).
If $|Y(x_1)|=3$, let $x_2 \in S_X$. By the intersection lemma, either $|Y(x_2)|=4$ and $Y(x_2)$ intersects $Y(x_1)$ in two points, or $|Y(x_2)|=3$. In the former case, $M(x_1)\cup M(x_2)$ is a $2$-net and there can not be another point in $S_X$. Hence we conclude as above.
In the latter case, we distinguish two sub-cases. If $Y(x_2)$ intersects $Y(x_1)$ in two points, then $M(x_1)\cup M(x_2)$ is a 2-net. Therefore a third point $x_3 \in S_X$ has the property that $Y(x_3)$ intersects $(Y(x_1)\cup Y(x_2))_Y$
in exactly two points, and we have eventually a 2-net, say $A$, with $A_Y =Y$, so there can not be another point in $S_X$. It is readily checked that 
$S$ is fully erasable, starting by $Y(x_3)$, then $Y(x_2)$ and $Y(x_1)$.
If $Y(x_2)$ intersects $Y(x_1)$ in one point, then a third point $x_3 \in S_X$ intersects either $Y(x_1)$ in two points and $Y(x_2)$ in one point or the contrary.
Therefore we have eventually a 2-net with full $Y$ projection, so there can not be another point in $S_X$. It is readily checked that 
$S$ is fully erasable, starting from the meshes $M(x_1)$ or $M(x_2)$ whose $Y$-projection have only one intersection point with $Y(x_3)$, then continuing with
$M(x_3)$ and $M(x_1)$.
\end{proof}

The following proposition is a slight complement to Theorem \ref{thm2net}. It describes a situation where we can conclude that the WEP holds for an increasing limit of sets $S_n$:

\begin{proposition}\label{prop:Sn}
Let $(S_n)_{n \geq 1}$ be an increasing sequence such that the WEP holds for each $S_n$, and let $S = \cup_{{n \geq 1}} S_n$. If for each $n \geq 1$,  any $S_n$-affine function is the restriction to $S_n$ of an 
$S_{n+1}$-affine function, then the WEP holds for $S$.
\end{proposition}

\begin{proof}
Let $f$ be a real-valued function on $S$ and for any $n\geq 1$ let $f_n$ be its restriction to $S_n$. Then there is some triplet $(\varphi_1, h_1, \psi_1)$ such that
$f_1(x,y)=\varphi_1(x)+ h_1(x)(y-x)-\psi_1(y)$ on $S_1$. Assume by induction that there exists a sequence of triples
$(\varphi_p, h_p, \psi_p)_{1 \leq p \leq n}$, for $n \geq 1$, such that $\varphi_{p+1|(S_p)_X}=\varphi_p$ for all $1 \leq p \leq  n-1$ if $n \geq 2$, and the same for the other two functions. Then
since the WEP holds for $S_{n+1}$,  we have for some triple $(\varphi_{n+1}^0,h_{n+1}^0,\psi_{n+1}^0 )$
\[f_{n+1}(x,y)=\varphi_{n+1}^0 (x)+ h_{n+1}^0(x)(y-x)-\psi_{n+1}^0 (y) \quad \text{on } S_{n+1}.\] In particular
\[ (\varphi_{n+1}^0-\varphi_{n})(x)+ (h_{n+1}^0-h_n)(x)(y-x)-(\psi_{n+1}^0-\psi_n)(y)=0\] on $S_n$ and $q_{n}:=\psi_{n+1|(S_n)_Y}^0-\psi_n$ is an $S_{n}$-affine function.
Let $t_{n+1}$ be an $S_{n+1}$-affine function whose restriction to $S_n$ is $q_{n}$. We have on $ S_{n+1}$
\[  t_{n+1}(y) = r_{n+1}(x) + s_{n+1}(x)(y-x) \]
for suitable functions $r_{n+1},s_{n+1}$. Defining
\[ \varphi_{n+1} :=\varphi_{n+1}^0-r_{n+1}, \quad 
h_{n+1} :=h_{n+1}^0-s_{n+1},\quad \psi_{n+1} :=\psi_{n+1}^0-t_{n+1},\] 
yields $\varphi_{n+1|(S_n)_X}=\varphi_{n+1|(S_n)_X}^0-(\varphi_{n+1|(S_n)_X}^0-\varphi_{n})(x)=\varphi_{n}$, and in the same way we get $h_{n+1|(S_n)_X}=h_n$ and $\psi_{n+1|(S_n)_Y}=\psi_n $.

It follows that the functions $\varphi :=\lim_{n\to \infty} \varphi_n$, $h:=\lim_{n\to \infty} h_n$ and $\psi :=\lim_{n\to \infty} \psi_n$ are well defined on the whole set $S$. Indeed, for any $(x,y)\in S$ there exists $k \geq 1$ such that $(x,y) \in S_k$. Hence $\lim_{n \to \infty} \varphi_n (x) = \varphi_k (x)$ and similarly for the other two limits. Finally, we have
$f(x,y)=\varphi(x)+h(x)(y-x)-\psi(y)$ on $S$.
\end{proof}

Finally, we have the following:

\begin{corollary}
Let $(S_n)_{n \geq 1}$ be an increasing sequence of 2-nets such that the WEP holds for each $S_n$, and let $S = \cup_{{n \geq 1}} S_n$. Then the WEP holds for $S$.
\end{corollary}
\begin{proof}
The result follows readily from Proposition \ref{prop:Sn} above and from the fact that any affine function on $(S_n)_Y,  n \geq 1$, is the restriction to $(S_n)_Y$ of the affine function defined on $(S_{n+1})_Y$ with the same slope and intercept.
\end{proof}

%%%%%%%%%%%%%%%%%%%%%%%
\section{Cycles and extremality}\label{cycles}
%%%%%%%%%%%%%%%%%%%%%%%

In this section we examine the relation between extremality of a measure $Q$ in $\mathcal M (\mu,\nu)$ and the existence of cycles in its support. 
This is motivated by a following well-known geometrical characterization of extremal probabilities (without the martingale property) with given marginals  that we have already mentioned in the introduction: let $\mu,\nu$ be given marginals with countable supports, then a measure $Q$ with marginals $\mu,\nu$ is extremal if and only if its support does not contain any cycle. For clarity and for later use, we recall the relevant notion of (classical) cycle in the following

\begin{definition}\label{cycle} Let $S \subset X \times Y$. A (classical) cycle $\mathcal C$ in $S$ is any finite sequence of paths $\mathcal C= (x_i,y_i)_{i=1}^{2n} \subset S$ with $n \ge 1$, such that: \begin{enumerate}
\item either $y_{2i}=y_{2i-1}$, $x_{2i+1}=x_{2i}$, $x_{2i} \neq x_{2i-1}$, and $y_{2i+1} \neq y_{2i}$, or the same condition with $x$ and $y$ interchanged;
\item $x_1=x_{2n}$ (in which case $y_1=y_2$) or $y_1=y_{2n}$ (in which case $x_1=x_2$) and $x_i \neq x_j$ and $y_i \neq y_j$ for $1\le i \le j-3 < 2n-3$. 
\end{enumerate}\end{definition}

\begin{notation}\label{notation}
\emph{Since a cycle is a sequence of paths, there is a natural order along the cycle. For a given path $(x_i, y_i)$ in the cycle, either $x_{i+1} \neq x_i$ and we will say
that $(x_i, y_i)$ is an {\it outgoing} path from $x_i$, or $x_{i+1} = x_i$ and we will say that it is an {\it incoming} path. We use the convention $x_{i+1}=x_1$ if $i=2n$.
By relabelling if necessary, we can assume without loss of generality that  $(x_1, y_1)$ is an outgoing path from $x_1$. Then we can enumerate the cycle starting from $x_1$ as follows: $(x_1, y_1), (x_2, y_2),\ldots ,(x_{2n}, y_{2n})$,
where $y_1=y_2$ and $x_{2n}=x_1$. We will ease the notation by denoting a cycle directly by the sequence of its points, where the last one coincide with the first one by convention:  
$(x_1,y_2,x_3, y_4,\ldots ,x_{2n})$. Note also that a cycle will be identified with its support and orientation: for instance the cycles $(x_1, y_1, x_2, y_3, x_1)$  and $(x_2, y_3, x_1, y_1, x_2)$ are the same
cycle, and $(x_1, y_3, x_2, y_1, x_1)$ has the same support, but opposite orientation.}
\end{notation}

Proofs of the equivalence between extremality of $Q$ and absence of cycles in the support of $Q$ can be found in \cite{letac,muk}. The main idea is that if such a cycle exists the measure $Q$ can be perturbed along that cycle while preserving the marginals as follows: let $\alpha >0$ be a given parameter, set
\begin{equation}\label{perturb} Q_1 (x_i,y_i) = Q(x_i,y_i) + (-1)^i \alpha , \quad  Q_2 (x_i,y_i) = Q(x_i,y_i) - (-1)^i \alpha , \quad 1 \le i \le 2n,\end{equation}
and $Q_1(x,y)=Q_2(x,y)=Q(x,y)$ otherwise. Hence, since $\alpha$ can be chosen sufficiently small so that $Q_k$, for $k=1,2$, are probability measures, we have $Q_k \in \mathcal P (\mu,\nu)$, $k=1,2$, and $Q=(Q_1+Q_2)/2$. Whence $Q$ is not extremal in $\mathcal P (\mu,\nu)$. \medskip

In this section we investigate to which extent this idea can be exploited in our martingale context. We will introduce first a very natural notion of cycles in our context, {\it cycles of 2-meshes} and we will end this section by a generalization of this notion in terms of classical cycles. \medskip

\subsection{Cycles of 2-meshes}
%%%%%%%%%%%%%%%%

Let us start by revisiting the proof of Lemma \ref{inter-extr} (Intersection Lemma under extremality), where we have constructed a perturbation of the initial probability $Q \in \mathcal M(\mu, \nu)$
in the subset $A=\{(x_i, y_j)\}_{i \in \{1,2\}, j \in \{1,2,3\}}$. It turns out that this perturbation can be seen in a different perspective. In fact, it can be viewed as a perturbation along a {\it cycle of 2-meshes}.
A given 2-mesh $M := \{(x_1,y_i): i=1,2\}$ can be clearly seen as an element of the product space $X \times Y^2$. With a slight abuse of notation we will sometimes write $M = (x_1; y_1,y_2)$. 
We define then a cycle of 2-meshes in a natural way:

\begin{definition}\label{cycle2mesh}
A cycle of 2-meshes is a cycle in $X \times Y^2$.
\end{definition}\noindent
Hence, the decomposition of the set $A$ in a cycle of 2-meshes is
\begin{equation}\label{A-cycle} (x_1; y_1, y_2),\, (x_2; y_1, y_2), \, (x_2; y_2, y_3), \, (x_1; y_2, y_3),\end{equation} or, using the notation \ref{notation} applied the to product space $X \times Y^2$,
$(x_1, (y_1,y_2), x_2, (y_2, y_3), x_1)$. Now the key observation is the following: associate to each 2-mesh $(x;y, y^\prime )$ a perturbation of total mass $\alpha$ dispatched as $p$ on the path $(x,y)$ and $q$ on the path $(y,z)$, so that $\alpha=p+q$, i.e.
\[ Q(x,y) + p, \quad Q(x,y^\prime)+ q, \quad p+q = \alpha,\]
for $p,q \in [0,1]$. In order for such a perturbation to preserve the martingale property, we impose 
\[ py+qy^\prime =0,\] giving 
\[ q = \frac{\alpha y}{y-y^\prime}, \quad p = \frac{-\alpha y^\prime }{y-y^\prime},\] 
so that given $\alpha$ there is a unique possible choice for $p,q$, which do not depend on the origin point of the 2-mesh.
Along the cycle (\ref{A-cycle}) of 2-meshes, in order to preserve the mass $\nu(y)$ at each point $y \in Y$, we choose the following
sequence of perturbations: $\alpha, -\alpha, \alpha, -\alpha$. This also grants that the total mass of the perturbation at each point in $X$ is zero. Choosing $\alpha$ small enough, such a procedure leads to a new probability measure, say $Q^\alpha \in \mathcal M(\mu,\nu)$. Finally, applying a perturbation with opposite signs, i.e. $-\alpha, \alpha, -\alpha, \alpha$, we get another probability, say $\tilde Q^\alpha \in \mathcal M(\mu,\nu)$, such that $Q=(Q^\alpha + \tilde Q^\alpha)/2$. This contradicts the extremality of $Q$ in $\mathcal M(\mu,\nu)$.

The construction is exactly the same for a cycle of 2-meshes of any finite length. Therefore, we can sum up what we have just obtained in the following, where we identify a point $(x,(y_1, y_2))$ in $X \times Y^2$
with the subset $\{(x,y_1), (x,y_2)\}$ of $X \times Y$.

\begin{proposition} [Perturbation along a cycle of 2-meshes]
Let $Q \in \mathcal M(\mu, \nu)$ be extremal. Then the support of $Q$ does not contains any cycle of 2-meshes.
\end{proposition}

A similar implication holds with the WEP replacing extremality:

\begin{proposition} 
Let $S$ be a subset of $X\times Y$ and assume that the WEP holds for $S$. Then $S$ does not contain any cycle of 2-meshes. 
\end{proposition}
\begin{proof}
Assume that the WEP holds for $S$ and that $S$ contains a cycle of 2-meshes $M^1,\ldots, M^n$ for some $n \ge 2$. Pick any function $f: S \to \mathbb R$ and let $(\varphi, h, \psi)$ be a decomposition of $f$ as in (\ref{WEPf}).
Within any 2-mesh $M^i = (x_i ;y_{i,1} , y_{i,2})$, one has 
\[ \frac{f(x_i,y_{i,2})-f(x_i,y_{i,1})}{y_{i,2}-y_{i,1}}=h(x_i)+\frac{\psi(y_{i,2})-\psi(y_{i,1})}{y_{i,2}-y_{i,1}}, \quad i=1,\ldots, n.\]
Summing up along the cycle of such 2-meshes, notice that the $h$ term (resp. the $\psi$ term) will cancel whenever consecutive 2-meshes have their $x$-points (resp. one or more of their $y$-points) in common. We get therefore the equality
$$0=\sum_i (-1)^i \frac{f(x_i,y_{i,2})-f(x_i,y_{i,2})}{y_{i,2}-y_{i,1}}.$$
Since the function $f: S \to \mathbb R$ is arbitrary, we get a contradiction.
\end{proof}

Therefore, absence of cycles of 2-meshes is necessary for both extremality and WEP. At this point it is very natural to ask if the converse statement is also true. Unfortunately, albeit being a natural notion to consider in a martingale setting, it turns out that it is not sufficient for neither extremality nor WEP as the following example shows. 

\begin{example} \label{cycles-counter} {\rm Let $X=\{x_1,x_2,x_3\}$ and $Y=\{y_i\}_{i=1,\ldots,5}$ be in decreasing order, i.e. $x_1 > x_2 > x_3 >0$ and $y_1 > \cdots > y_5>0$. Moreover, we assume $y_2 = x_1$ and $y_4 < x_3 < x_2 < y_3$. Consider the following pattern:
\begin{figure}[H] 	
\centering
\includegraphics[scale=1]{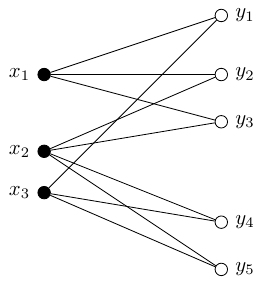}
\end{figure}
\noindent This cannot be the support of an extremal martingale measure with given marginals, since 
the WEP does not hold. Moreover, it can also be checked by direct inspection that it does not contain any cycles of 2-meshes.} \end{example}

Thus, this seems to imply that the  notion of cycle of 2-meshes is too strong. 
Finding a more general pattern, admitting a perturbation which preserves the marginals and the martingale property, is the topic of the next section.

\subsection{Perturbation along a pool of cycles}
%%%%%%%%%%%%%%%%%%%%%%%%%%%%

First, we observe that if a probability measure is not extremal in $\mathcal M(\mu, \nu)$, it is also not extremal in $\mathcal P (\mu, \nu)$, so there will be a classical
cycle in its support (cf. \cite{letac,muk}). We know that a one-parameter perturbation is naturally attached to a cycle as in (\ref{perturb}), in such
a way that the marginals are preserved. There is no hope for this single perturbation to preserve {\it also} the martingale property: Indeed the martingale
condition at any point $x$ in the $X$-section of the cycle reads $\pm \alpha (y_1-y_2)=0$ where $y_1$ and $y_2$ are the two distinct points in $Y$
such that the paths $(x,y_1)$ and $(x,y_2)$ belong to the cycle.
The idea we are going to exploit is that combining many cycles in a suitable way should add sufficiently many degrees of freedom to the perturbation in order to fulfil the martingale property. 

\subsubsection{Revisiting cycles of 2-meshes}\label{sec:twoMeshesRevisited}

As a warm-up, we revisit the notion of cycle of 2-meshes in terms of classical cycles. Consider a cycle of 2-meshes $M_1,\ldots,M_{2n}$, for some $n \ge 1$, as in Definition \ref{cycle2mesh}. 
Each consecutive pair of 2-meshes $M_i,M_{i+1}$ with the same $Y$-section can be viewed as a classical cycle $\mathcal C_i$ (with length 4), so that
a cycle of 2-meshes with length $2n$ can be seen as a set of $n$ classical cycles, each with length 4. If we attach a sufficiently small perturbation parameter $\alpha_i >0$ to each of these cycles 
as in (\ref{perturb}) the marginals will be preserved. 

We use the following notation for the points of any cycle $\mathcal C_i$: its left-hand points are given by $(\mathcal C_i)_X =\{x_i, x_{i+1}\}$, 
while the right-hand points are $(\mathcal C_i)_Y = \{y_{i,1},y_{i,2}\}$. Notice that the cycle property implies in particular that $y_{i,2} = y_{i+1,1}$ for all $i$.

Let us check if the martingale property is also preserved under the perturbation given by the parameters $\alpha_i$ as above. We start from a left-hand point $x_1$ and a cycle $\mathcal C_1$. The other left-hand point of $\mathcal C_1$ will be $x_2$, which is also left-hand point $x_2$ of the cycle $\mathcal C_2$. 
The martingale property at $x_2$ is fulfilled if and only if 
\[-\alpha_1 (y_{1,1}-y_{2,1})+\alpha_{2} (y_{2,1}-y_{2,2})=0,\]
and similarly for the other points. Eventually we obtain the martingale condition at the point $x_1$ as 
\[ -\alpha_{n} (y_{n,1}-y_{n,1})+\alpha_{1} (y_{1,1}-y_{2,1})=0.\]
The key observation is that this last equation, {\it because of the cycle property}, is obtained as a sum of the $n-1$ previous ones. Indeed: 
\begin{eqnarray*}0 &=& \sum_{i=1}^n \alpha_i  ((y_{i,1}-y_{i,2}) -(y_{i,1}-y_{i,2})) \\
&=& -\alpha_1 (y_{1,1}-y_{2,1})+\alpha_2 (y_{2,1}-y_{2,2})+ \cdots +(-\alpha_{n} (y_{n,1}-y_{n,1})+\alpha_{1} (y_{1,1}-y_{2,1})).\end{eqnarray*}
We get therefore a system of $n-1$ equations with $n$ unknowns, which is readily solved by induction in this case, taking for instance $\alpha_1$ as free parameter. Therefore, we have obtained a perturbation preserving both the marginals and the martingale property.

\subsubsection{Generalization to arbitrary cycles}

We can now generalize the previous pattern to cycles of any length, in the following way: consider $n$ classical cycles $\mathcal C_i$ with $i=1,\ldots,n$, each of arbitrary length, with the property that the union of the $X$-sections of the cycles contains exactly
$n$ distinct points $x_1, \ldots ,x_n$, i.e. $\bigcup_{i=1}^n ({\mathcal C_i})_X = \{x_1,\ldots,x_n\}$. Let $\gamma_{i,j} = y_{i,j} - y_{i,{j+1}}$ be the difference between the right-hand point of the outgoing path from $x_j$ and the right-hand point of the incoming path to $x_j$ along the cycle $\mathcal C_i$. 

We start with a statement relating the WEP and a certain pattern of cycles for some subset $S \subset X \times Y$:

\begin{proposition}\label{cyclesGen}
Assume that a set $S$ in $X \times Y$ contains, for some $n \geq 2$, a set of $n$ classical cycles $\mathcal C_i$ such that:
\begin{enumerate}
\item $|\bigcup_{i=1}^n ({\mathcal C_i})_X|=n$;
\item the cycles are \emph{free}, i.e. each cycle $\mathcal C_i$ contains a path which does not belong to any other cycle $\mathcal C_k$, $k \neq j$.
\end{enumerate}
Then the WEP does not hold for $S$.
\end{proposition}

\begin{proof}
Assume that WEP$(f)$ holds for any function $f$. Let $h_j$ be the coefficient of $(y-x)$ in the WEP decomposition attached to the point $x_j$. We have along each cycle $\mathcal C_i$, with the notations above:
\[ \sum_j h_j \gamma_{i,j} = \hat f_i\]
where $\hat f_i$ is the sum of the values of $f$ along the paths of the cycles, counted with a positive sign if path goes from $X$ to $Y$ along the cycle,
and with a negative sign otherwise. By the cycle property we have $\sum_j \gamma_{i,j}=0$ so the the identity vector belongs to the kernel of the matrix $\Gamma = (\gamma_{i,j})$. By the rank theorem,
the image of the matrix $\Gamma$ is of dimension at most $n-1$. It remains to observe that the assumptions 2 in the statement (i.e. cycles are free) implies that the vectors $(\hat f_i)_{1 \leq i \leq n}$ when $f$ varies in the set of all real-valued functions defined on $X \times Y$, span the whole $\mathbb R^n$.
Therefore, for functions $f$ such that the vector $\hat f_i$ does not belong to the image of $\Gamma$, the above relation does not hold, whence a contradiction.
\end{proof}

Let us now go back to the construction of a martingale perturbation.  Attach to each cycle $\mathcal C_i$ a perturbation as in (\ref{perturb}) with parameter $\alpha_i$. So a (classical) perturbation 
associated to the vector $\alpha_1, \ldots,\alpha_n$ can be built along the $n$ cycles $\mathcal C_1,\ldots, \mathcal C_n$ by choosing sufficiently small parameters $\alpha_i$. 

Let us investigate now, exactly as above, the martingale conditions at the left-hand points $x_j$. The contribution of the cycle $\mathcal C_i$ to the martingale condition at point $x_j$ will be $\alpha_i \gamma_{i,j}$ where the classical cycle condition entails $\sum_j \gamma_{i,j} =0$, for each $i$, and the martingale condition at the point $x_j$ reads $\sum_i \alpha_i \gamma_{i,j}=0$. Exactly as in Section \ref{sec:twoMeshesRevisited} we have therefore
\begin{equation} \label{eq:martCond}
0 = \sum_i \alpha_i \sum_j \gamma_{i,j} = \sum_j \sum_i \alpha_i \gamma_{i,j}
\end{equation}
so that the martingale condition at any point $x_j$ is entailed by the martingale conditions at all the other left-hand points.

We are left with $n-1$ equations for $n$ unknown, and by the rank theorem the solution is a vector space of dimension at least $1$, so that by taking a sufficiently small element in this space we 
get a perturbation preserving the martingale property. Now it remains to prove that this perturbation is not zero.\footnote{A simple example where this would happen is given by a set of two cycles 
with twice the same cycle: our approach would lead to a single equation in two unknowns $\alpha_1, \alpha_2$, with a one dimensional solution space given by $\alpha_1+\alpha_2=0$. 
The resulting perturbation in this case is the sum of the perturbations $\alpha_1$ and $\alpha_2$ along the cycle, hence the zero perturbation.}
So we need an additional hypothesis, which is given by assumption 2 in the previous proposition (i.e. freeness of cycles): it guarantees indeed than any non-zero solution vector of \ref{eq:martCond} is associated to a 
non-zero perturbation, since for each index $i$ there is a path which is perturbed by $\alpha_i$ only, and not by a linear combination of the components of $\alpha$. 
We have just proved the following

\begin{proposition}\label{cyclesGen-extr}
Let $Q \in \mathcal M(\mu, \nu)$. Assume that the support of $Q$ satisfies the assumptions of Propostion \ref{cyclesGen} with free cycles $\mathcal C_1,\ldots, \mathcal C_n$.
Then $Q$ is not extremal.
\end{proposition}

\begin{example}
{\rm The pattern in Example \ref{cycles-counter} satisfies the hypotheses of Propositions \ref{cyclesGen} and \ref{cyclesGen-extr}: the three (classical) cycles can be taken as 
\[ (x_1, y_2, x_2, y_3, x_1), \quad (x_2, y_4, x_3, y_5, x_2), \quad (x_1, y_1, x_3, y_4, x_2, y_2, x_1).\] }
\end{example}
It can also be checked than in all the finite examples of extremal points stated in this paper there are
at most $n-1$ free cycles with $n$ left-hand points. We leave the converse statement, i.e. if $Q$ is not extremal then there is necessarily such a 
configuration of cycles in its support, as a conjecture.

\section{Conclusion}\label{conclusion}
%%%%%%%%%%%%%%%%%%%%
In this paper, motivated by the recent literature in model-free finance, we have investigated the properties of the supports of extremal martingale measures with given marginals. Using the Douglas-Lindenstrauss-Naimark Theorem, we have provided an equivalence between extremality of some martingale measure $Q$ with given marginals and the denseness in $L^1(Q)$ of a suitable linear subspace, which has a natural financial interpretation as the set of all semi-static strategies. Furthermore, we have studied the combinatorial properties of the supports of such extremal measures in the countable case. More precisely, we have focused on a pointwise version of the weak PRP, called WEP, which implies the extremality when one of the two marginals has finite support. Then we have introduced three combinatorial properties called ``full erasability'', \ref{2LP} and ``no deadlocks'', and we have proved the following implications (among others):\vspace{0.5cm}
\begin{figure}[H]	
\centering
\includegraphics[scale=1]{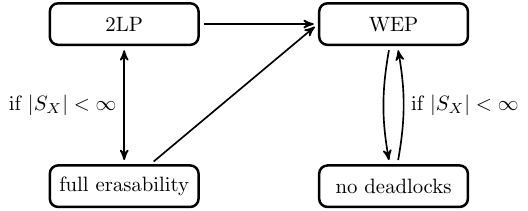}
\end{figure}
\vspace{0.3cm}

\noindent Moreover, we have also started to study the role of cycles in relation to extremality and identified some forbidden patterns, generalizing the notion of (classical) cycles, in the supports of extremal measures. Many examples have been provided in order to illustrate all those notions and how they differ from each other. Many problems remain open, such as showing the equivalence between the WEP and the extremality in full generality (if it holds), the relation with graph theory and, more importantly, to what extent those implications can be extended to the non-countable case, e.g. when the marginals have absolutely continuous densities. They are all left for future research.

\end{document}